%% file: rotations.tex
\documentclass[11pt]{amsart}
\usepackage{pinlabel}
\usepackage{amsmath, amsthm, amssymb, amscd, mathrsfs, eucal}
\usepackage{enumerate}
\usepackage[T1]{fontenc}
\usepackage{hyperref}
\usepackage{float}
\usepackage{mathtools}
\usepackage{color}
\usepackage[margin=1in]{geometry}
\usepackage{graphicx}
\usepackage{adjustbox}

\newtheorem{theorem}{Theorem}[section]
\newtheorem{lemma}[theorem]{Lemma}
\newtheorem{corollary}[theorem]{Corollary}

\newtheorem{proposition}[theorem]{Proposition}
\newtheorem{example}[theorem]{Example}

\newtheorem*{claim*}{Claim}

\theoremstyle{definition}

\newtheorem{convention}[theorem]{Convention}
\newtheorem{assumption}[theorem]{Assumption}

\theoremstyle{theorem}

\theoremstyle{remark}

\usepackage{color}
\usepackage{pdfcolmk}

\newcounter{jcomments}

\newcounter{acomments}

\newcommand{\rom}[1]{\text{\uppercase\expandafter{\romannumeral #1\relax}}}

\def\Z{\mathbb Z}
\def\N{\mathbb N}
\def\R{\mathbb R}

\def\defeq{\vcentcolon=}

\def\Hbb{\mathbb{H}}

\def\new{\mathrm{new}}

\numberwithin{equation}{section}

\usepackage{pdfpages}
\title{Irrational rotations and $2$-filling rays}
\author{Lvzhou Chen}
\address{Department of Mathematics\\ Purdue University\\ West Lafayette, IN, USA}
\email[L.~Chen]{lvzhou@purdue.edu}

\author{Alexander J. Rasmussen}
\address{Department of Mathematics \\ Stanford University \\ Stanford, CA, USA}
\email[A.J.~Rasmussen]{ajrasmus@stanford.edu}

\begin{document}

\begin{abstract}
We study a skew product transformation associated to an irrational rotation
of the circle $[0,1]/\sim$. This skew product keeps track of the number of times an orbit of the rotation lands in the two complementary intervals of $\{0,1/2\}$ in the circle. 
We show that
under certain conditions on the continued fraction expansion of the irrational
number defining the rotation, the skew product transformation has certain dense orbits.
This is in spite of the presence of numerous non-dense orbits. We use this to construct
laminations on infinite type surfaces with exotic properties. In particular, we
show that for every infinite type surface with an isolated planar end, there is
an \emph{infinite} clique of $2$-filling rays based at that end. These $2$-filling rays
are relevant to Bavard--Walker's \emph{loop graphs}.
\end{abstract}
\maketitle

\section{Introduction}
Our goal in this paper is to study skew products over irrational rotations on the circle
and to explore relationships to laminations on infinite type surfaces. In particular, we
prove that specific orbits are dense in a collection of skew product transformations.
%prove that a collection of skew product systems each have a dense orbit.
We use this to
show that certain laminations on infinite type surfaces have dense boundary leaves.
Finally, we use this to construct certain rays on infinite type surfaces with exotic
properties, which are relevant to the study of Bavard--Walker's \emph{loop graphs}.

We consider the circle $S^1$ as the closed unit interval $[0,1]$ with $0$ and $1$ identified.
For a number $\alpha \in [0,1)$, we define the rotation $t=t_\alpha:S^1\to S^1$ by
$t(x)=x+\alpha$ modulo 1. We define the function $f:S^1\to \R$ by
$f=\chi_{[0,1/2)} -\chi_{[1/2,1)}$ where $\chi_E$ denotes the characteristic
function of the set $E$ in question. We define a resulting skew product transformation
$T=T_\alpha:S^1\times \Z \to S^1\times \Z$ by
\[
	T(x,s) \vcentcolon = (tx, s+fx) = (x+\alpha, s+fx).
\]
We endow $\Z$ with the discrete topology and $S^1\times \Z$ with the resulting
product topology. We consider the continued fraction expansion for $\alpha$,
\[
	\alpha=[0;a_1,a_2,\ldots] = \frac{1}{a_1 + \frac{1}{a_2 + \frac{1}{a_3 + \ldots}}}.
\]
We prove:

\begin{theorem}
	\label{theorem:skewprod}
	Suppose that the continued fraction expansion $\alpha=[0;a_1,a_2,\ldots]$ satisfies that
	$a_1\geq 5$ is odd and $a_n\geq 6$ is even for every $n>1$. Then, for any $s\in \Z$,
	the (forward) orbit $\{T^n(1/2,s)\}_{n=0}^\infty$ is dense in $S^1\times \Z$.
\end{theorem}

Note that for any $n\geq 0$, $T^n(x,s)=(t^nx, s+S_n(x))$ where
$S_n(x)=\sum_{i=0}^{n-1} f(t^i x)$ is the $n^{\text{th}}$ Birkhoff sum for $x$.
For any $m\in \Z$, we consider the set
$\Sigma(x,m) \vcentcolon = \{n \in \Z_{\geq 0} : S_n(x)=m\}$ of times $n$ at which
$S_n(x)$ is equal to $m$. Denote by $k+\Sigma(x,m)$ the set above translated by $k\in\Z$.
The following corollary, when $k=0$, is a restatement of Theorem
\ref{theorem:skewprod}. The general case, which is useful for our applications, also quickly follows from Theorem \ref{theorem:skewprod}; see the next section for its proof.

\begin{corollary}\label{cor: main}
	Suppose that the continued fraction expansion $\alpha=[0;a_1,a_2,\ldots]$ satisfies that
	$a_1\geq 5$ is odd and $a_n\geq 6$ is even for every $n>1$. Then, for any $k,m\in \Z$, the
	partial orbit $\{t^n(1/2)\}_{n\in k+\Sigma(1/2,m)}$ is dense in $S^1$.
\end{corollary}

In particular, for any $m\in\Z$, there are infinitely many $n\ge0$ with $S_n(1/2)=m$.
In contrast, it is shown in \cite[Theorem 1]{Heavy} (by the characterization of $\mathcal H_{2}$) that 
$S_n(1/2)<0$ for all $n\ge 1$ if $\alpha=[0; a_1,a_2,\ldots]$ with $a_i$ even for all $i$ odd.
In addition, for almost every $\alpha$, there is an uncountable set (with Hausdorff dimension equal to some constant $c\in(0,1)$ independent of $\alpha$) of initial points $x\in[0,1)$ with $S_n(x)\le0$ for all $n\ge1$ \cite{1/2-heavy}. 

We use our results above to construct examples of interesting laminations and rays on infinite type
surfaces. For the first statement, recall that a complete hyperbolic surface $X$ is of the
\emph{first kind} if it is equal to its convex core. A geodesic lamination $\Lambda$ on $X$
is \emph{topologically transitive} if it contains a leaf which is dense in $\Lambda$.

\begin{theorem}
	\label{theorem:lamination}
	Let $S$ be any orientable infinite type surface with at least one isolated puncture. 
	Then there is a hyperbolic surface $X$ of the first kind homeomorphic to $S$, and
	a geodesic lamination $\Lambda$ on $X$, such that $\Lambda$ is topologically transitive,
	with infinitely many leaves which are not dense in $\Lambda$.
\end{theorem}

For our second application, we consider the \emph{loop graph} $L(S;p)$ of an infinite type
surface $S$ with an isolated puncture $p$, defined by Bavard in \cite{RayGraph} and studied
further by Bavard--Walker in \cite{GromBd} and \cite{Simult}. The vertices of $L(S;p)$ are the
simple, essential loops on $S$ asymptotic to $p$ on both ends, considered up to isotopy.
Two isotopy classes are joined by an edge when the corresponding isotopy classes can be
realized disjointly.
The graph $L(S;p)$ is Gromov-hyperbolic and of infinite diameter \cite{Simult}; see also \cite{Aramayona}. 
Bavard--Walker \cite{Simult} identified the points on the Gromov boundary of $L(S;p)$ with cliques of
the so-called high-filling rays.
As a related notion, a \emph{$2$-filling ray} $\ell$ on $S$ is a kind of
\emph{fake boundary point} for $L(S;p)$. Namely, such a ray is asymptotic to $p$,
and intersects every loop on $S$, so that it has strong filling properties similar to high-filling rays,
but it is not high-filling.
See Section \ref{sec:background} for the precise definitions.

Bavard--Walker asked in \cite[Question 2.7.7]{GromBd} whether $2$-filling rays exist, for instance, when $S$ is the plane minus a Cantor set. This was answered affirmatively by the authors in \cite{2Fill}. Such $2$-filling rays always come organized into families of mutually disjoint $2$-filling rays called \emph{cliques}.
The authors showed that the cliques can have any finite cardinality in \cite[Theorem 5.1]{2Fill}, and asked
whether such cliques can be infinite \cite[Question 5.7]{2Fill}. We answer this question affirmatively in Theorem \ref{theorem:infiniteclique} below for any infinite type surface $S$ with an isolated puncture. In particular, $2$-filling rays exist on all such surfaces.
The analogous problem about the size of cliques of high-filling rays has been solved by methods different from our dynamical approach: 
such a clique can be of any finite cardinality on any infinite type surface $S$ with an isolated puncture by \cite[Theorem 8.1.3]{GromBd},
and it can also be infinite at least when $S$ is the plane minus a Cantor set by \cite{infinite_clique}.

\begin{theorem}\label{theorem:infiniteclique}
	Let $S$ be an orientable infinite type surface with at least one isolated puncture $p$. 
	Then there exists an infinite clique of $2$-filling rays on $S$ based at $p$.
\end{theorem}

It is an open problem to describe the boundaries of the loop graphs $L(S;p)$ as spaces of geodesic 
laminations. The authors believe that solving this problem would lead to
significantly better understanding of the graphs $L(S;p)$. The existence of exotic
laminations and rays as constructed in Theorem \ref{theorem:lamination} and
Theorem \ref{theorem:infiniteclique} and in \cite{2Fill} point to the difficulty
of solving this problem and to the complexity of the graphs $L(S;p)$. It would be
interesting to use skew products to construct other interesting laminations and
mapping classes of infinite type surfaces.

\subsection*{Acknowledgments}
We owe a debt of gratitude to Jon Chaika for teaching us and suggesting the method 
used in this paper for studying irrational rotations. We also thank the anonymous
referee for their careful reading of the paper and suggestions. The second author
was partially funded by NSF grant DMS-2202986.

\section{Proof of Theorem \ref{theorem:skewprod}}

We choose $\alpha=[0;a_1,a_2,\ldots]$ satisfying the conditions of Theorem
\ref{theorem:skewprod}; i.e. $a_1\geq 5$ is odd and $a_i\geq 6$ is even for every
$i\geq 2$. Furthermore we set $\alpha_1=\alpha$ and for $i\geq 2$,
\begin{equation}\label{eqn: alpha_i}
	\alpha_i = [0;a_i-1, a_{i+1}, a_{i+2}, \ldots].
\end{equation}
Let
\[
G(x) = \frac{1}{x} - \left\lfloor \frac{1}{x}\right\rfloor
\]
be the Gauss
transformation. Then $\alpha_{i+1} = G(\alpha_i) / (1-G(\alpha_i))$.

Our method of proof considers first return maps to certain subintervals, which shares some similarity with the renormalization procedure used in related work; see \cite{1/2-heavy} for instance, which also gives insights about the behavior of other orbits.

We will compute a sequence of nested intervals
$[0,1)=I_1 \supset I_2 \supset I_3 \supset \ldots$ each centered at $1/2$ and the first return maps to $I_i$. Let
\[
	k_i:I_i \to \N, \ \ \ k_i(x) = \inf\{k > 0 : t^kx \in I_i\}
\] be the first return time to $I_i$ and
\[
	\overline{t}_i:I_i \to I_i, \ \ \ \overline{t}_i(x) = t^{k_i(x)}(x)
\] be the first return map.
Our construction guarantees the following properties, which we will verify later.
\begin{enumerate}
	\item $\overline{t}_i$ is rotation by $(-1)^{i+1}\alpha_i$ (rescaled by the length of $I_i$).
	\item Moreover, we compute the induced Birkhoff sums
	      \[
		      \overline{f}_i:I_i \to \Z, \ \ \ \overline{f}_i(x) = \sum_{j=0}^{k_i(x) - 1} f(t^jx);
	      \] i.e. $\overline{f}_i$ records the Birkhoff sum accumulated before a point in $I_i$ returns
	      to $I_i$ under iteration of $t$. Then by our construction $\overline{f}_i$ will be equal to $+1$ on the
	      sub-interval of points to the left of $1/2$ and $-1$ on the sub-interval of points to the
	      right of $1/2$.
\end{enumerate}

Theorem \ref{theorem:skewprod} is a consequence of the following, seemingly weaker proposition.

\begin{proposition}\label{prop:firstreturns}
	There is a sequence of intervals $[0,1)=I_1\supset I_2\supset I_3\supset \ldots$ such
	that:
	\begin{enumerate}
		\item $I_i$ contains $1/2$ for each $i$ and is symmetric about $1/2$ for each $i$;\label{item: symmetric}
		\item for each $i\geq 1$ the interval $I_{i+1}$ has length
		      $|I_{i+1}|\le \alpha_{i}|I_{i}|$;\label{item: length bound}
		\item for each $i\geq 2$, after rescaling $I_i$ by $1/|I_i|$,
		      the function $\overline{f}_i(x)$ is equal to $\chi_{[0,1/2)} - \chi_{[1/2,1)}$;\label{item: new f}
		\item for any $i$, and for any $m\in \Z$, there exists an orbit point
		      $t^k(1/2) \in I_i$, for some $k\in \Z_+$,
		      with $S_k(1/2)=m$.\label{item: near 1/2}
	\end{enumerate}
\end{proposition}

\begin{proof}[Proof of Theorem \ref{theorem:skewprod} assuming Proposition
		\ref{prop:firstreturns}]

	First we improve the last bullet point to the following claim: for any $m\in \Z$
	there exist orbit points $t^k(1/2)$ in $I_i$ \emph{to the right of} $1/2$ with $S_k(1/2)=m$
	and similarly there exist points $t^k(1/2)$ \emph{to the left of} $1/2$ with $S_k(1/2)=m$, $k\in\Z_+$.
	We focus on the case of finding points to the right of $1/2$, as the other case is analogous.

	Choose $i$ odd, so that the first return to $I_i$ is rotation by
	$\alpha_i=[0;a_{i}-1,a_{i+1},a_{i+2},\ldots]$. For any $m\in \Z$, there is some $k\in\Z_+$ such that
	$t^k(1/2)\in I_{i+1}$ with $S_k(1/2)=m$. If $t^k(1/2)$ lies to the right of $1/2$
	then there is nothing to show. Otherwise, since $a_i-1 \geq 5$, the length of
	$I_{i+1}$ is at most $\alpha_{i}|I_i|$, and $t^k(1/2)$ lies in $I_{i+1}$, we have that
	$\overline{t}_i(t^k(1/2)) , \overline{t}_i^2(t^k(1/2)) \in I_i$ both lie to the right of $1/2$. 
	Now we compute the Birkhoff sum at $\overline{t}_i^2(t^k(1/2))$.
	Let
	\[
		l=k_i(t^k(1/2)) + k_i(\overline{t}_i(t^k(1/2)))
	\]
	be the second return time of $t^k(1/2)$ to $I_i$. Then by (\ref{item: new f})
	we have
	\[
		S_{k+l}(1/2) = S_k(1/2) + \overline{f}_i(t^k(1/2)) +
		\overline{f}_i(\overline{t}_i(t^k(1/2))) = S_k(1/2)+1+(-1)= S_k(1/2) = m.
	\]
	That is, the point $t^{k+l}(1/2)=\overline{t}_i^2(t^k(1/2)) \in I_i$ justifies the claim.

	Now the theorem follows from this claim.
	By Proposition \ref{prop:firstreturns}, the closure of the
	orbit of $(1/2,0)$ contains $\{1/2\}\times \Z$.
	By the claim, for any $m\in \Z$, we may choose points $t^k(1/2)$ arbitrarily close to $1/2$ and \emph{to the right} with $S_k(1/2)=m$.
	Consider any $\epsilon\in(0,1/2)$ and a
	point $(x,m)\in S^1 \times \Z$. We want to show that for any $s\in \Z$, the orbit
	of $(1/2,s)$ contains a point in $[x,x+\epsilon)\times \{m\}$. Since $\alpha$ is
	irrational, the rotation $t$ is minimal and there exists $k\geq 0$ with
	$t^k (1/2) \in \left[x,x+\frac{\epsilon}{2}\right)$. Suppose that $S_k(1/2)=N$.
	The functions $\{f\circ t^i\}_{i=0}^k$ are individually constant on a short interval
	that has $1/2$ as its left endpoint, so there is $0<\delta<\epsilon/2$
	such that any point $y \in [1/2,1/2 + \delta)$ satisfies $S_k(y)=S_k(1/2)=N$.
	By the claim, we can choose $l \geq 0$ such that
	\[
	T^l(1/2,s) = (t^l(1/2), S_l(1/2))\in \left[\frac{1}{2},\frac{1}{2}+\delta\right)
	\times \{m-s-N\}.
	\]
	Then
	\[
	T^{l+k}(1/2,s) = \big(t^{l+k}(1/2), s+S_l(1/2) + S_k(t^l(1/2))\big).
	\]
	As $t^l(1/2)\in [1/2,1/2+\delta)$ and $t^k(1/2)\in[x,x+\epsilon/2)$,
	we have
	\[
	t^{l+k}(1/2)=t^k(t^{l}(1/2))\in [t^k(1/2), t^k(1/2)+\delta)\subset [x,x+\epsilon).
	\]
	In addition, $S_k(t^l(1/2))=N$ by our choice of $\delta$.
	It follows that $s+S_l(1/2) + S_k(t^l(1/2))=s+(m-s-N)+N=m$ and
	$T^{l+k}(1/2,s)\in [x,x+\epsilon)\times \{m\}$, as desired.
\end{proof}

Now we deduce Corollary~\ref{cor: main} from Theorem~\ref{theorem:skewprod}.
\begin{proof}[Proof of Corollary~\ref{cor: main}]
	Theorem~\ref{theorem:skewprod} is equivalent to the following major case of Corollary~\ref{cor: main}: For any $m\in\Z$, the partial orbit $\{t^n(1/2)\}_{n\in \Sigma(1/2,m)}$ is dense in $S^1$. For the general case, for an arbitrary $k\in \Z$, we are interested in the density of the orbit points $t^{n+k}(1/2)$ with $n\in \Sigma(1/2,m)$, i.e. the image of the partial orbit $\{t^n(1/2)\}_{n\in \Sigma(1/2,m)}$ under the rotation $t^k$. Such a partial orbit is also dense in $S^1$.
\end{proof}

It remains to find the intervals $I_i$ and prove Proposition \ref{prop:firstreturns}. For this we proceed by induction.
To construct $I_{i+1}$ based on $J=I_i$ and its first return map, the inductive step fits into the following setup:
\begin{assumption}\label{assump: basic setup for J}
	\leavevmode
	\begin{itemize}
		\item We have chosen an interval $J\subset [0,1)$ which contains $1/2$ and is centered at $1/2$.
		\item After scaling $J$ by $1/|J|$ to unit length, the first return map to $J$, which we denote by $t_J$, is a rotation by a number $\beta=\pm[0;b,\ldots]$ with $b\geq 5$ odd (so $|\beta|<1/5$).
	\end{itemize}
\end{assumption}
We construct a sub-interval $J^\new$ of $J$ that is centered at $1/2$ with well-understood first return map among other properties.
We describe the construction below in Lemmas \ref{lemma: inductive step, beta>0} and \ref{lemma: inductive step, beta<0},
depending on the sign of $\beta$.

In the discussion below, we frequently look at different left-closed and right-open sub-intervals of $[0,1)$ centered at $1/2$ and rescale them to length $1$.
To avoid confusion due to different scales, we use the following convention.
\begin{convention}
	For a sub-interval $J$ of $[0,1)$ centered at $1/2$, we abuse notation and
	let $J:[0,1)\to J$ be the unique affine homeomorphism fixing $1/2$.
	Then for any $x\in(0,1)$, $J(x)$ is the point at distance $x$ from the left
	endpoint of $J$ after rescaling $J$ to unit length. Similarly, $J[a,b)$ is the
	sub-interval of $J$ corresponding to the interval $[a,b)\subset [0,1)$
	after rescaling $J$ to unit length.
\end{convention}

\bigskip

\noindent \underline{First case: $\beta > 0$}

\bigskip

We first consider the case $\beta>0$ and introduce some notation in order to state the inductive construction in Lemma \ref{lemma: inductive step, beta>0}.
Note that the first coefficient $b=\lfloor 1/\beta \rfloor$. We partition
$J$ into sub-intervals
\[
	J_0 = J[0,\beta), \ J_1 = J[\beta, 2\beta), \ \ldots,
	\ J_{b-1} = J[(b - 1)\beta, b\beta), \ J_{b} = J[b\beta, 1),
\]
each of which has length $\beta|J|$ except for $J_{b}$, which has length $\beta G(\beta)|J|$
where $G(x) = 1/x - \lfloor 1/x \rfloor$ as before.
For a point $x\in J$, we have $t_J(x)=t^{k(x)}(x)$ where $k(x) = \inf \{k > 0 : t^kx \in J\}$, and by our induction hypothesis, $t_J(J(x))=J(x+\beta\mod 1)$ and $t_J(J_i)=J_{i+1}$ for all $0\le i<b-1$. We consider the orbit of $x\in J$ under $t$ before its first return to $J$, and record the sequence of values of $f$ along this orbit, namely
\[
	\mathcal F(x) = \{ f(x), f(t(x)), \ldots, f(t^{k(x)-1}(x))\}.
\]
This is equivalent to recording the sequence of partial sums $\mathcal{S}(x)=\{S_i(x)\}_{i=1}^{k(x)}$ with $S_i(x)\defeq\sum_{j=0}^{i-1} f(t^j(x))$. The partial sums keep track of the increment in the second coordinate (compared to $(x,m)$) along the orbit of $(x,m)$ under $T$ in the skew product:
\[
	\left\{(x,m),\ T(x,m)=\left(tx, m+fx\right), \ \ldots, \
	T^{k(x)}(x,m)=\left(t^{k(x)}x, \ m+S_{k(x)}(x)\right)\right\}
\]
Finally, we set $\Sigma(x) = S_{k(x)}(x)$, which is the total sum of the sequence $\mathcal{F}(x)$.
We use $\mathcal{F}_1\cdot\mathcal{F}_2$ to denote the concatenation of two sequences $\mathcal{F}_1$ and $\mathcal{F}_2$.

Here are our remaining assumptions for the case $\beta>0$ in addition to Assumption \ref{assump: basic setup for J}:
\begin{assumption}\label{assump: beta>0}
	There are sequences $\mathcal F_+$, $\mathcal F_-$, and $\mathcal F_0$ with total sums $1$, $-1$, and $0$, respectively,
	such that
	\begin{itemize}
		\item Whenever $x\in J[0,1/2)$, we have $\mathcal F(x) = \mathcal F_+$, 
		\item Whenever $x\in J[1/2, 1 - \beta)$,
		      we have $\mathcal F(x) = \mathcal F_-$, 
		\item Whenever $x\in J[1-\beta,1)$,
		      we have $\mathcal F(x) = \mathcal F_- \cdot \mathcal F_0$.
	\end{itemize}
	Here we allow $\mathcal F_0$ to be an empty sequence.
\end{assumption}

As a consequence of the assumptions above, the sequence $\mathcal{S}(x)$ must be the sequence of partial sums for $\mathcal F_+$, $\mathcal F_-$, or $\mathcal F_+\cdot \mathcal{F}_0$ depending on the location of $x$ as above. Denote the partial sum sequences of $\mathcal F_+$ and $\mathcal F_-$ by $\mathcal S_+$ and $\mathcal{S}_-$, and denote the total sums of $\mathcal F_+$, $\mathcal F_-$, $\mathcal{F}_0$ as $\Sigma_+, \Sigma_-, \Sigma_0$. The assumptions above imply $\Sigma_+=1$, $\Sigma_-=-1$, and $\Sigma_0=0$.

We record the maximum and minimum over each sequence of partial sums, i.e.
\[
	m_+ \defeq \min\mathcal S_+, \ M_+ \defeq \max \mathcal S_+,
	\ m_- \defeq \min \mathcal S_-, \ M_- \defeq \max \mathcal S_-
\]

Our aim is to find a sub-interval
$J^\new \subset J$ containing and centered at $1/2$, for which the first return
to $J^\new$, rescaled by $1/|J^\new|$, is a rotation by a new number
$\beta^\new = -[0;c,\ldots]$ with $c\geq 3$ determined by $\beta$ explicitly as in Lemma \ref{lemma: inductive step, beta>0} below.
Moreover, for $x\in J^\new$, denote the first return time to $J^\new$ as
\[
	k^\new(x)=\inf \{k > 0 : t^k(x) \in J^\new\}
\]
and consider as before the sequence of $f$-values
\[
	\mathcal F^\new(x) \defeq \{f(x), f(t(x)),\ldots,f(t^{k^\new(x)-1}(x))\}.
\]
Let $\mathcal S^\new(x)$ be the sequence of partial sums associated to $\mathcal F^\new(x)$,
and let $\Sigma^\new(x) = S_{k^\new(x)}(x)$ be the total sum.

The following lemma shows how we construct the sub-interval $J^\new$ and the nice properties guaranteed by the construction.

\begin{lemma}\label{lemma: inductive step, beta>0}
	Suppose there is a sub-interval $J\subset [0,1)$ with first return map $t_J$ satisfying Assumptions \ref{assump: basic setup for J}
	and \ref{assump: beta>0}
	with $\beta=[0;b,c+1,\ldots]>0$, where $c\ge 3$.
	Denote $b = 2n+1$ with $n\geq 1$.
	Then the sub-interval $J^\new\subset J$ given by
	\[
		J^\new \defeq J\left[\frac{1}{2} - \frac{1}{2}\beta(1-G(\beta)), \frac{1}{2} + \frac{1}{2}\beta(1-G(\beta))\right).
	\]
	has the following properties:
	\begin{enumerate}
		\item $J^\new $ is symmetric about $1/2$ of length $\beta (1-G(\beta))|J|$.\label{item: J^new prop, sym and len}
		\item $J^\new$ is a sub-interval of $J_n$, the right endpoints of $J^\new$ and $J_n$ are the same,
		      and the left endpoint of $J^\new$ has distance $\beta G(\beta)|J|=|J_b|$ from the left endpoint of $J_n$.\label{item: J^new prop, rel J_n}
		\item The image of $J_n\setminus J^\new$ under $n+1$ iterations of $t_J$ is $J_b$.\label{item: J^new prop, orbit}
		\item Re-scaling by $1/|J^\new|$, the first return map to $J^\new$ is rotation
		      by \label{item: J^new prop, 1st return}
		      \[
			      \beta^\new \defeq -G(\beta) / (1-G(\beta))=-[0;c,\ldots].
		      \]
		\item There are sequences
		      \[
			      \mathcal F_+^\new\defeq \mathcal F_+ \cdot \mathcal F_-^n \cdot
			      \mathcal F_0 \cdot \mathcal F_+^n,\quad
			      \mathcal F_-^\new \defeq \mathcal F_-^{n+1} \cdot \mathcal F_0
			      \cdot \mathcal F_+^n,\quad\text{ and }\quad
			      \mathcal F_0^\new \defeq \mathcal F_+ \cdot \mathcal F_-^{n+1}\cdot \mathcal F_0
			      \cdot \mathcal F_+^n
		      \]
		      satisfying:
		      \begin{itemize}
			      \item whenever $x\in J^\new[0,\beta^\new)$ we have $\mathcal F^\new(x) =
				            \mathcal F_+^\new \mathcal F_0^\new$,
			      \item whenever $x\in J^\new[\beta^\new, 1/2)$ we have $\mathcal F^\new(x) =
				            \mathcal F_+^\new$,
			      \item whenever $x\in J^\new[1/2, 1)$ we have $\mathcal F^\new(x) = \mathcal F_-^\new$.
		      \end{itemize}
		      Moreover, $\mathcal F_+^\new, \mathcal F_-^\new, \mathcal F_0^\new$
		      have total sums $1,-1,0$ respectively. \label{item: J^new prop, const sequences and expression}
	\end{enumerate}
\end{lemma}

\begin{proof}
	Item (\ref{item: J^new prop, sym and len}) is immediate.
	To see item (\ref{item: J^new prop, rel J_n}), note that $1/2$ lies in the interval $J_n = J[n\beta, (n+1)\beta)$ and its distances to the endpoints are
	\[
		\left(\frac12 - n\beta\right)|J| = \frac12(1-2n\beta)|J| = \frac12 \beta|J|(1 + G(\beta)) \text{ and } \left((n+1)\beta-\frac12\right)|J|=\frac12 \beta|J|(1 - G(\beta)).
	\]
	Since $t_J$ is rotation by $\beta|J|=|J_i|$ for $i<b$ and $t_J(J_i)=J_{i+1}$ for any $i<b-1$, item (\ref{item: J^new prop, orbit}) easily follows; see Figure \ref{figure: J-decomp, beta>0}.
	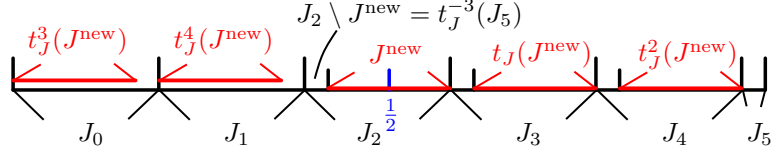
\begin{figure}
		\labellist
		\small \hair 2pt
		\pinlabel $\color{blue}\frac{1}{2}$ at 143 2
		\pinlabel $J_0$ at 30 -5
		\pinlabel $J_1$ at 85 -5
		\pinlabel $J_2$ at 135 -5
		\pinlabel $J_3$ at 195 -5
		\pinlabel $J_4$ at 250 -5
		\pinlabel $J_5$ at 282 -5
		\pinlabel $J_2\setminus J^\new=t_J^{-3}(J_5)$ at 150 40
		\pinlabel $\color{red}t_J^3(J^\new)$ at 25 30
		\pinlabel $\color{red}t_J^4(J^\new)$ at 80 30
		\pinlabel $\color{red}J^\new$ at 145 25
		\pinlabel $\color{red}t_J(J^\new)$ at 200 25
		\pinlabel $\color{red}t_J^2(J^\new)$ at 255 25
		\endlabellist
		\centering
		\includegraphics{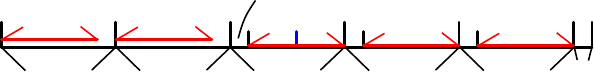}
		\caption{The decomposition of $J$ into $J_i$'s when $b=2n+1=5$ for $n=2$ and the orbit of $J^\new$ under iterations of $t_J$, for $\beta>0$}\label{figure: J-decomp, beta>0}
	\end{figure}

	Now we analyze the first return map. After scaling by $1/|J_n|=1/\beta$, the first return map to $J_n$ is rotation by $(b\beta-1)/\beta=-G(\beta)$. Therefore, by restricting to the further sub-interval $J^\new$ and
	rescaling by $1/|J^\new|$, one can check that the first return map to $J^\new$ is rotation by $\beta^\new =-G(\beta)/(1-G(\beta))$. This is essentially a simple case of Rauzy--Veech induction. See the next several paragraphs for more details.
 A direct computation verifies item (\ref{item: J^new prop, 1st return}), i.e.
	\[
		\beta^\new = -G(\beta) / (1-G(\beta))=-[0; c,\ldots].
	\]

	Next we compute the sequences $\mathcal F_+^\new, \mathcal F_-^\new, \mathcal F_0^\new$.
	Note that by item (\ref{item: J^new prop, rel J_n}) $t_J^k(J^\new)$ is a sub-interval of $J_{n+k}$ sharing its right endpoint for $0\le k\le n$ and $t_J^k(J^\new)$ is a sub-interval of $J_{k-(n+1)}$ sharing its left endpoint for $n+1\le k\le 3n+1$; see Figure \ref{figure: J-decomp, beta>0}.

	In particular, $t_J^{2n+1}(J^\new)$ lies in $J_n$ sharing its left endpoint, and we observe that this completes the first return to $J^\new$ by $t_J^{2n+1}$ for $x\in J^\new[\beta^\new, 1)$. Counting for which $0\le k\le 2n$ we have $t_J^{k}(x)$ on the left or right of $J(1/2)$,
	we observe that, for $x\in J^\new[\beta^\new,1/2)$, the sequence $\mathcal F^\new(x)$ is equal to $\mathcal F_+^\new$ defined as in (\ref{item: J^new prop, const sequences and expression}), i.e.
	\[
		\mathcal F_+^\new= \mathcal F_+ \cdot \mathcal F_-^n \cdot
		\mathcal F_0 \cdot \mathcal F_+^n;
	\]
	and for $x\in J^\new[1/2,1)$, the sequence $\mathcal F^\new(x)$ is equal to $\mathcal F_-^\new$ as in (\ref{item: J^new prop, const sequences and expression}), i.e.
	\[
		\mathcal F_-^\new = \mathcal F_-^{n+1} \cdot \mathcal F_0
		\cdot \mathcal F_+^n.
	\]

	On the other hand, any $x\in J^\new[0,\beta^\new)$ also returns to $J_n$ for the first time via $t_J^{2n+1}$ but lands in $J_n\setminus J^\new=t_J^{2n+1}(J^\new[0,\beta^\new))$. After another $2n+2$ 
	iterations of $t_J$, $x$ finally returns to $J^\new$ and the additional sequence of $f$-values is $\mathcal F_0^\new$ as in (\ref{item: J^new prop, const sequences and expression}), i.e.
	\[
		\mathcal F_0^\new =\mathcal F_+ \cdot \mathcal F_-^{n+1}\cdot \mathcal F_0
		\cdot \mathcal F_+^n.
	\]
	Indeed, after the first return to $J_n$ (i.e. $2n+1$ iterations of $t_J$), $x$ lands to the right of $1/2$ for the next $n+1$ iterations of $t_J$ in $J$ rather than $n$ times as before, since $t_J^{n+1}(J_n\setminus J^\new)=J_b$. Finally, the last $n+1$ iterations take such $x$ back to $J^\new$ and $x$ stays on the left of $1/2$ until it is back.

	Therefore, for $x\in J^\new[0,\beta^\new)$, we see that $\mathcal F^\new(x)$ is the concatenation $\mathcal F_+^\new \cdot \mathcal F_0^\new$ as claimed in item (\ref{item: J^new prop, const sequences and expression}). The computations above in these three cases together verify item (\ref{item: J^new prop, const sequences and expression}), where the total sums of the sequences $\mathcal F_+^\new, \mathcal F_-^\new, \mathcal F_0^\new$ are $1,-1, 0$, respectively, as an immediate corollary of the expressions in (\ref{item: J^new prop, const sequences and expression}) and the total sums of $\mathcal F_+, \mathcal F_-,
		\mathcal F_0$ given in Assumption \ref{assump: beta>0}.
\end{proof}

Now consider the partial sum sequences $\mathcal S_+^\new$ and $\mathcal S_-^\new$ for the sequences $\mathcal F_+^\new$ and $\mathcal F_-^\new$, respectively. We estimate the upper and lower bounds of these partial sum sequences:
\[
	M_+^\new \defeq \max \mathcal S_+^\new, \ m_+^\new \defeq
	\min \mathcal S_+^\new, \ M_-^\new \defeq \max \mathcal S_-^\new,
	\ m_-^\new \defeq \min \mathcal S_-^\new.
\]
\begin{lemma}\label{lemma: bounds for beta>0}
	For the sequences $\mathcal F_+^\new$, $\mathcal F_-^\new$ and the integer $n$ defined as in Lemma \ref{lemma: inductive step, beta>0}, assuming the total sums of $\mathcal F_+$, $\mathcal F_-$, and $\mathcal F_0$ to be $1, -1, 0$ respectively as in Assumption \ref{assump: beta>0}, and assuming $n\ge2$ (i.e. $b\ge5$), we have
	\begin{itemize}
		\item $M_+^\new \geq M_+$;
		\item $m_+^\new \leq m_- - (n-2)$;
		\item $M_-^\new \geq M_-$;
		\item $m_-^\new \leq m_- - n$;
	\end{itemize}
\end{lemma}
\begin{proof}
	These easily follow by inspection and the fact that $\Sigma_+ = +1, \Sigma_- = -1$.
	As $\mathcal F^\new _+$ starts with the sequence $\mathcal F_+$, we note that $\mathcal S_+$ is a prefix of the sequence $\mathcal S^\new_+$, which verifies the first bullet. The third bullet follows similarly.

	For the second bullet, consider the expression
	$$\mathcal F^\new_+=(\mathcal{F}_+ \cdot \mathcal{F}_-^{n-1})\cdot\mathcal F_- \cdot(\mathcal F_0 \cdot \mathcal F_+^n).$$
	The sequence $\mathcal F_+ \cdot \mathcal F_-^{n-1}$ has total sum $\Sigma_+ + (n-1)\Sigma_-=-(n-2)$, so for the subsequence $\mathcal F_-$ after these terms, its partial sum sequence $\mathcal{S}_-$ shifted by $-(n-2)$ appears as a subsequence of $\mathcal S^\new_-$, which implies the second bullet.

	The last bullet can be shown analogously, as the sequence $\mathcal F^\new_-$ starts with $(\mathcal F_-)^n\cdot \mathcal F_-$, where the part in parentheses has total sum $n\Sigma_-=-n$.
\end{proof}

\bigskip

\noindent \underline{Second case: $\beta < 0$}

\bigskip

We now consider the case $\beta < 0$. Denote $\gamma = -\beta$.
Then the first return map to $J$ is $t_J(J(x))=J(x-\gamma\mod 1)$.
The case here is essentially just mirroring the case above, as now we are rotating to the left. For clarity, we include some details below. Define the sequences $\mathcal F(x)$ and $\mathcal S(x)$ as before for any $x\in J$ and let $\Sigma(x)$ be the total sum of $\mathcal F(x)$. Here are the remaining assumptions for the case $\beta<0$.
\begin{assumption}\label{assump: beta<0}
	There are sequences $\mathcal F_+,\mathcal F_-,\mathcal F_0$ with total sums $\Sigma_+=1,\Sigma_-=-1,\Sigma_0=0$, respectively, such that
	\begin{itemize}
		\item For $x\in J[0,\gamma)$ we have $\mathcal F(x) = \mathcal F_+ \cdot \mathcal F_0$.
		\item For $x\in J[\gamma,1/2)$ we have $\mathcal F(x) = \mathcal F_+$.
		\item For $x\in J[1/2,1)$ we have $\mathcal F(x) = \mathcal F_-$.
	\end{itemize}
\end{assumption}

We denote the first coefficient by $b=\lfloor 1 / \gamma \rfloor$, and express it as $b=2n+1$ for some $n\geq 2$.
We again partition $J$ into intervals
\[
	J_0 = J[1-\gamma, 1), J_1 = J[1-2\gamma, 1-\gamma), \ldots, J_b = J[0, 1-b\gamma),
\]
and we have $t_J(J_i)=J_{i+1}$ for $0\le i < b-1$.
The interval $J_b$ has length $\gamma G(\gamma)$.

Our aim is again to find a sub-interval $J^\new \subset J$ containing $1/2$ and symmetric about $1/2$, for which the first return to $J^\new$ inherits nice properties regarding the sequence $\mathcal F^\new(x)$ of $f$-values and the sequence $\mathcal S^\new(x)$ of partial sums, defined just as in the previous case.

\begin{lemma}\label{lemma: inductive step, beta<0}
	Suppose there is a sub-interval $J\subset [0,1)$ with first return map $t_J$ satisfying Assumptions \ref{assump: basic setup for J}
	and \ref{assump: beta<0}
	with $\beta=-\gamma=-[0;b,c+1,\ldots]<0$, where $c\ge 3$.
	Denote $b = 2n+1$ with $n\geq 1$.
	Then with notation as above,
	the sub-interval $J^\new\subset J$ given by
	\[
		J^\new \defeq J\left[\frac{1}{2} - \frac{1}{2}\gamma(1-G(\gamma)), \frac{1}{2} + \frac{1}{2}\gamma(1-G(\gamma))\right).
	\]
	has the following properties:
	\begin{enumerate}
		\item $J^\new $ is symmetric about $1/2$ of length $\gamma(1-G(\gamma))|J|$.\label{item: J^new prop beta<0, sym and len}
		\item $J^\new$ is a sub-interval of $J_n$, the left endpoints of $J^\new$ and $J_n$ are the same, and the right endpoint of $J^\new$ has distance $\gamma G(\gamma)|J|=|J_b|$ from the right endpoint of $J_n$.\label{item: J^new prop beta<0, rel J_n}
		\item The image of $J_n\setminus J^\new$ under $n+1$ iterations of $t_J$ is $J_b$.\label{item: J^new prop beta<0, orbit}
		\item Re-scaling by $1/|J^\new|$, the first return map to $J^\new$ is rotation
		      by \label{item: J^new prop beta<0, 1st return}
		      \[
			      \beta^\new \defeq G(\gamma) / (1-G(\gamma))=[0;c,\ldots].
		      \]
		\item There are sequences
		\[
		\mathcal F_+^\new \defeq \mathcal F_+^{n+1} \cdot \mathcal F_0 \cdot
			      \mathcal F_-^n, \quad
			      \mathcal F_-^\new \defeq \mathcal F_- \cdot \mathcal F_+^n \cdot
			      \mathcal F_0 \cdot \mathcal F_-^n, \quad\text{and}\quad
			      \mathcal F_0^\new \defeq \mathcal F_- \cdot \mathcal F_+^{n+1} \cdot
			      \mathcal F_0 \cdot \mathcal F_-^n
			      \]
			      with total sums $1,-1, 0$, respectively, satisfying:
		      \begin{itemize}
			      \item whenever $x\in J^\new[0,1/2)$ we have $\mathcal F^\new(x) =
				            \mathcal F_+^\new$, 
			      \item whenever $x\in J^\new[1/2, 1-\beta^\new)$ we have $\mathcal F^\new(x) =
				            \mathcal F_-^\new$, 
			      \item whenever $x\in J^\new[1-\beta^\new, 1)$ we have $\mathcal F^\new(x) =
				            \mathcal F_-^\new \mathcal F_0^\new$.
		      \end{itemize}\label{item: J^new prop beta<0, const sequences}
	\end{enumerate}
\end{lemma}
\begin{proof}
	The proof is almost the same as that of Lemma \ref{lemma: inductive step, beta>0}, by symmetry; see Figure \ref{figure: J-decomp, beta<0}.
	So we just summarize a few key points below.
	\begin{figure}
		\labellist
		\small \hair 2pt
		\pinlabel $\color{blue}\frac{1}{2}$ at 143 2
		\pinlabel $J_0$ at 260 -5
		\pinlabel $J_1$ at 203 -5
		\pinlabel $J_2$ at 150 -5
		\pinlabel $J_3$ at 93 -5
		\pinlabel $J_4$ at 40 -5
		\pinlabel $J_5$ at 6 -5
		\pinlabel $J_2\setminus J^\new=t_J^{-3}(J_5)$ at 190 40
		\pinlabel $\color{red}t_J^3(J^\new)$ at 265 30
		\pinlabel $\color{red}t_J^4(J^\new)$ at 210 30
		\pinlabel $\color{red}J^\new$ at 145 25
		\pinlabel $\color{red}t_J(J^\new)$ at 90 25
		\pinlabel $\color{red}t_J^2(J^\new)$ at 35 25
		\endlabellist
		\centering
		\includegraphics{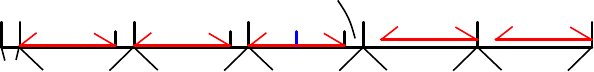}
		\caption{The decomposition of $J$ into $J_i$'s when $b=2n+1=5$ for $n=2$ and the orbit of $J^\new$ under iterations of $t_J$, for $\beta<0$.}\label{figure: J-decomp, beta<0}
	\end{figure}
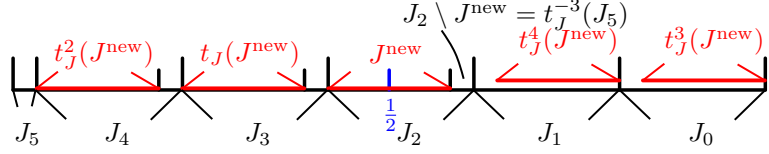
	Items (\ref{item: J^new prop beta<0, sym and len})--(\ref{item: J^new prop beta<0, orbit}) are just direct computations as before, noting that $1/2$ lies in the interval $J_n$ but $|J_b|$-closer to its \emph{left} endpoint this time.

	As in the previous case, after scaling by $1/|J_n|=1/\gamma$, the first return map to $J_n$ is rotation by $(1-b\gamma)/\gamma=G(\gamma)$, which is now positive.
	Thus, by restricting further to the sub-interval $J^\new$ and rescaling by $1/|J^\new|$ instead, the first return to $J^\new$ is rotation by $\beta^\new= G(\gamma) / (1-G(\gamma))$ as in item (\ref{item: J^new prop beta<0, 1st return}).

	Item (\ref{item: J^new prop beta<0, const sequences}) follows by an analysis of first returns to $J^\new$, which is just mirroring the case of $\beta>0$: the interval $t_J^{2n+1}(J^\new)$ lies in $J_n$ sharing the \emph{right} endpoint, completing the first return to $J^\new$ for all $x\in J^\new[0,1-\beta^\new)$, and the sequence of $f$-values depends on whether $x$ lies on the left or right of $1/2$, which only changes the first term ($\mathcal F_\pm$) in the concatenation. For those $x\in J^\new[1-\beta^\new,1)$, it takes another $2n+2$ iterations of $t_J$ to return to $J^\new$, resulting in the additional sequence $\mathcal{F}_0^\new$.
\end{proof}

As before, for the partial sum sequences $\mathcal S_+$ and $\mathcal S_-$ denote
\[
	M_+ \vcentcolon = \max \mathcal S_+, \ m_+ \vcentcolon = \min\mathcal S_+,
	\ M_- \vcentcolon = \max \mathcal S_-, \ m_- \vcentcolon = \min\mathcal S_-,
\]
and similarly for the partial sum sequences $\mathcal S^\new_+$ and $\mathcal S^\new_-$
by adding superscripts everywhere in the above equations.

The proof of the following lemma is similar to that of Lemma \ref{lemma: bounds for beta>0}, using the expressions for $\mathcal F_+^\new$ and $\mathcal F_-^\new$ in Lemma \ref{lemma: inductive step, beta<0}.

\begin{lemma}\label{lemma: bounds for beta<0}
	For the sequences $\mathcal F_+^\new$, $\mathcal F_-^\new$, and the integer $n$ defined as in Lemma \ref{lemma: inductive step, beta<0}, assuming the total sums of $\mathcal F_+$, $\mathcal F_-$, and $\mathcal F_0$ to be $1, -1, 0$, respectively, as in Assumption \ref{assump: beta<0}, and assuming $n\ge2$, we have
	\begin{itemize}
		\item $M_+^\new \geq M_+ + n$;
		\item $m_+^\new \leq m_+$;
		\item $M_-^\new \geq M_+ + (n-2)$;
		\item $m_-^\new \leq m_-$.
	\end{itemize}
\end{lemma}

Finally we can prove Proposition \ref{prop:firstreturns}.

\begin{proof}[Proof of Proposition \ref{prop:firstreturns}]
	We inductively construct $I_i$ and check the first three items as follows. For the base case $i=1$, set $I_1=[0,1)$ and $\beta_1=\alpha$. Then the three items are either obvious or vacuous.
	Since the first coefficient of $\alpha$ is odd, Assumptions \ref{assump: basic setup for J} and \ref{assump: beta>0} hold for the rotation $t_J=t$ on $J=I_1$, where $\mathcal{F}_+=\{1\}$, $\mathcal{F}_-=\{-1\}$, and $\mathcal{F}_0=\emptyset$ is the empty sequence.
	By Lemma \ref{lemma: inductive step, beta>0}, we have a sub-interval $I_2\defeq J^\new$ symmetric about $1/2$ of length less than $\alpha_1|I_1|$, for which the first return map is rotation by $\beta^\new=-\alpha_2$, where $\alpha_2$ is defined in formula (\ref{eqn: alpha_i}).

	Note that the first return map on $I_2$ with the sequences $\mathcal{F}^\new_+$, $\mathcal{F}^\new_-$, and $\mathcal{F}^\new_0$ from Lemma \ref{lemma: inductive step, beta>0} satisfies Assumptions \ref{assump: basic setup for J} and \ref{assump: beta<0}. Thus Lemma \ref{lemma: inductive step, beta<0} produces a sub-interval $I_3$ symmetric about $1/2$ of length less than $\alpha_2|I_2|$, for which the first return map is rotation by $\alpha_3$, with new sequences of $f$-values satisfying Assumptions \ref{assump: basic setup for J} and \ref{assump: beta>0}.

	We continue this process to define $I_i$ inductively, alternating between applications of Lemma \ref{lemma: inductive step, beta>0} and Lemma \ref{lemma: inductive step, beta<0} to $I_i$ with $i$ odd and even, respectively.
	Namely, given $I_i$ and $\alpha_i$, define $I_{i+1}=I_i^\new$ and $\alpha_{i+1}=\alpha_i^\new$.
	Item (\ref{item: J^new prop beta<0, const sequences}) in Lemmas \ref{lemma: inductive step, beta>0} and \ref{lemma: inductive step, beta<0} ensures item (\ref{item: new f}) in Proposition \ref{prop:firstreturns}.

	Define inductively $\mathcal{F}_+^i$, $\mathcal{F}_-^i$, $\mathcal{F}_0^i$ as the sequences of $f$-values for first returns to $I_i$, using $\mathcal F_\pm^{i+1}=(\mathcal F_\pm^i)^\new$ and similarly for $\mathcal F_0^{i+1}$. Let $\mathcal S_\pm ^i$ be the sequences of partial sums of $\mathcal{F}_\pm^i$.
	Let $M_\pm^i$ and $m_\pm^i$ be the bounds on the partial sums $\mathcal S_\pm ^i$ estimated in Lemmas \ref{lemma: bounds for beta>0} and \ref{lemma: bounds for beta<0}.

	Finally we prove (\ref{item: near 1/2}). 
	We first prove below that there are points $t^k(1/2)\in I_1$ in
	the forward orbit of $1/2$ with $S_k(1/2)$ equal to any given integer $m$. We will explain at the end how to find such points in $I_j$ for any $j\in\Z_+$ instead of $I_1$.
	Note that the sequence $\mathcal S_-^i$ consists exactly of the Birkhoff sums of $1/2$ that occur before $1/2$
	returns to $I_i$ for the first time. We have $M_-^i = \max \mathcal S_-^i$ and $m_-^i = \min\mathcal S_-^i$. Thus, it suffices to prove that $M_-^i\to +\infty$ and $m_-^i\to -\infty$ as $i\to \infty$.

	Since the first coefficient $a_i-1$ (or $a_1$ when $i=1$) of $\alpha_i$ is odd and at least $5$ by assumption, $n_i\defeq (a_i-2)/2\ge2$ (and $n_1\defeq (a_1-1)/2\ge2$). We have $m_-^{2i}\le m_-^{2i-1}-n_{2i-1}\le m_-^{2i-1}-2$ by Lemma \ref{lemma: bounds for beta>0} and $m_-^{2i+1}\le m_-^{2i}$ by Lemma \ref{lemma: bounds for beta<0}. It follows that $m_-^{i+2}\le m_-^i-2$ for all $i$ and hence $\lim m_-^i=-\infty$.
	For a similar reason, $\lim M_+^i=+\infty$, which we now use to deduce that $\lim M_-^i=+\infty$.
	In fact, we have $M_-^{2i+2}\ge M_-^{2i+1}\ge M_+^{2i}+(n_{2i}-2)\ge M_+^{2i}$ by Lemmas \ref{lemma: bounds for beta>0} and \ref{lemma: bounds for beta<0}.
	Thus $\lim M_-^i=+\infty$ and $\lim m_-^i=-\infty$ as claimed.

    The proof above works in the same way after replacing $\alpha$ by $\alpha_j$, $I_1$ by $I_j$, $f$ by $\overline{f}_j$, and $t$ by $\overline{t}_j$ for any $j\in\Z_+$. That is, there is $k\in\Z_+$ such that $\overline{t}^k_j(1/2)\in I_j$ with $\overline{t}_j$-Birkhoff sum equal to $m$. As $\overline{t}_j$ is the first return map to $I_j$, such a point in $I_j$ is also in the forward orbit of $1/2$ under $t$, and its $\overline{t}_j$-Birkhoff sum is equal to the corresponding $t$-Birkhoff sum, which completes the proof.
\end{proof}

The same method can be used to study the Birkhoff sum along other orbits. We give a sketch for one explicit example below, which we use later to find a leaf that is not dense in Theorem \ref{theorem:lamination}.
\begin{example}\label{example: non dense full orbit}
    Fix $m\ge2$. Let $\alpha=[0; 2m+1, 2m+2, 2m+2, \cdots]$ (i.e. $a_1=2m+1$ and $a_n=2m+2$ for all $n\ge2$), which satisfies the assumption of Theorem \ref{theorem:skewprod}. 
    We consider the orbit of $x=(1+\alpha)/2$ and claim that $S_n(x)\le 0$ for all $n\in \Z$. Here we set $S_0(x)=0$ and $S_{-n}(x)=-\sum_{k=1}^n f(t^{-k}(x))$ for any $n>0$ so that $T^n(x,s)=(t^n(x), s+S_n(x))$ for all $n\in\Z$. The claim implies that the (forward and backward) orbit of $(x,0)$ under iterations of $T$ always has non-positive second coordinate.
    
    We sketch a proof of the claim. First, note that we can take care of the backward orbit by symmetry. In fact, for our particular $x$ we have $t^{-(n+1)}(x)=1-t^n(x)$ for all $n>0$, i.e. the backward orbit (starting at $t^{-1}(x)=(1-\alpha)/2$) and the forward orbit (starting at $x$) are symmetric around $1/2$, and thus the sequences of $f$-values along the forward and backward orbits differ by a negative sign. It follows that $S_{-n}(x)=S_n(x)$ for all $n\in\Z_+$. So it suffices to check that $\max_{n\ge1} S_n(x)\le 0$.

    To compute $S_n(x)$ with $n>0$, we use the same renormalization procedure with the nested intervals $I_1\supset I_2\supset \cdots$ as above.
    Let $\mathcal F^i_{\pm}$ and $\mathcal F^i_{0}$ (resp. $\mathcal S^i_{\pm}$ and $\mathcal S^i_{0}$) be the sequence of $f$-values (resp. partial sums) defined inductively as in the proof above. Let $M^i_{\pm}=\max \mathcal S^i_{\pm}$ and $M^i_0=\max \mathcal S^i_0$. 
    A direct computation shows that $t^m(x)=1-\frac{1}{2}\alpha G(\alpha)$ and $t^{2m+1}(x)=(m+1)\alpha-\frac{1}{2}\alpha G(\alpha)$, so the forward orbit enters $I_2$ for the first time after $2m+1$ iterations of $t$. In $I_2$-coordinates, we have $t^{2m+1}(x)=I_2(y)$ with 
    $$y=\frac{[(m+1)\alpha - \frac{1}{2} \alpha G(\alpha)] - [m\alpha+\alpha G(\alpha)]}{\alpha (1-G(\alpha))}=\frac{\alpha - \frac{3}{2} \alpha G(\alpha)}{\alpha (1-G(\alpha))} = 1 - \frac{1}{2}\beta,$$
    where $\beta=G(\alpha)/(1-G(\alpha))$ and the first return map $\bar{t}_2:I_2\to I_2$ is rotation by $-\beta$ in $I_2$-coordinates by Lemma \ref{lemma: inductive step, beta>0}. Our choice of $\alpha$ makes $\beta=\alpha$.
    Then applying the first return map $\bar{t}_2$ another $m$ times we
    arrive at $I_2(1-(m+1/2)\beta)$, at which point we land in $I_3$ for the first time. In $I_3$-coordinates, this is $I_3(z)$ with
    $$z = \frac{[1 - (m+1/2)\beta] - [1 - (m+1)\beta]}{\beta(1-G(\beta))} = \frac{1}{2}(1+\gamma),$$ 
    where $\gamma = G(\beta)/(1-G(\beta))$. 
    Noting that $\gamma=\beta=\alpha$ by our choice of $\alpha$, we see $z=x$, so are are now exactly at $I_3(x)$, 
    and the first return map to $I_3$ is rotation by $\gamma=\alpha$ in $I_3$-coordinates by Lemma \ref{lemma: inductive step, beta<0}. 
    Thus from here on the analysis repeats. It follows that the sequence of $f$-values along the forward orbit is given by
    \begin{equation}\label{eqn: f-val seq}
        [(\mathcal F^1_-)^{m+1}\cdot (\mathcal F^1_+)^{m} \cdot (\mathcal F^2_-)^{m}]\cdot [(\mathcal F^3_-)^{m+1}\cdot (\mathcal F^3_+)^{m} \cdot (\mathcal F^4_-)^{m}]\cdots
    \end{equation}

    The Birkhoff sums are the partial sums of this sequence, and to analyze them we compute $M_{\pm}^k$ and $M_0^k$ for all $k\ge1$. The idea of the computation is similar to the proof of Lemmas \ref{lemma: bounds for beta>0} and  \ref{lemma: bounds for beta<0}, which yields the following recursive formulas for our particular $\alpha$:
    $$M^{2k}_+=M^{2k-1}_+,\quad M^{2k}_-=M^{2k-1}_-,\quad M^{2k}_0=M^{2k-1}_+,$$
    and
    $$M^{2k+1}_+=M^{2k}_0+m+1,\quad M^{2k+1}_-=M^{2k}_0+m-1,\quad M^{2k+1}_0=M^{2k}_0+m,$$
    for all $k\ge1$.
    Then by induction, we have
    $$M^{2k-1}_+=(m+1)k-m, \quad M^{2k-1}_-=(m+1)k-m-2, \quad M^{2k-1}_0=(m+1)k-m-1,$$
    for all $k\ge2$ and
    $$M^{2k}_+=(m+1)k-m,\quad M^{2k}_-=(m+1)k-m-2,\quad M^{2k}_0=(m+1)k-m,$$
    for all $k\ge1$.
    Now by examining the sequence in formula (\ref{eqn: f-val seq}) bracket by bracket, it is straightforward to check that $\max_{n\ge1} S_n(x)=-1$.
\end{example}

\section{Background on laminations and rays}
\label{sec:background}

We recall some background on geodesic laminations and geodesic rays on hyperbolic surfaces.
Let $X$ be a complete oriented hyperbolic surface without boundary. We will typically
consider the case that $X$ is of the \emph{first kind}. This means that the limit set of
$\pi_1(X)$ acting on the universal cover $\widetilde X \cong \mathbb{H}^2$ is the
entire Gromov boundary $\partial X\cong S^1$.
A \emph{geodesic lamination}
$\Lambda$ on $X$ is a closed subset of $X$ consisting of pairwise disjoint, simple,
complete geodesics. Each such complete geodesic is called a \emph{leaf} of
$\Lambda$.

In Section \ref{sec:laminations}, we will construct laminations on hyperbolic
surfaces using train tracks, weight systems, and foliated rectangles. Here is
the necessary background.
A \emph{train track} $\tau$ on $X$ is a locally
finite graph embedded on $X$ with the following
additional structure. At any vertex $v$ of $\tau$, the set $\mathcal B(v)$ of edges incident to $v$ has a circular order induced from the orientation of $X$. We have a partition of $\mathcal B(v)$ into
a pair of non-empty sets $\mathcal B_i(v)$ and $\mathcal B_o(v)$
which we call incoming and outgoing, respectively, such that the total order defined by any $g\in \mathcal B_o(v)$ (resp. $g\in \mathcal B_i(v)$) restricts to the same order on $\mathcal B_i(v)$ (resp. $\mathcal B_o(v)$) independent of $g$. 
For $e,f \in \mathcal B_i(v)$, we write $e<f$ if and only if $(e,f,g)$ is counterclockwise at $v$ for any $g\in \mathcal B_o(v)$. 
For $e,f \in \mathcal B_o(v)$, we write $e<f$ if and only if $(e,f,g)$ is
clockwise at $v$ for any $g\in \mathcal B_i(v)$.
The edges of $\tau$ are called
\emph{branches} and the vertices of $\tau$ are called \emph{switches}. A \emph{train path} $t$
on $\tau$ is a finite or infinite path
immersed in $\tau$ with the property that at every switch
$v$, $t$ enters $v$ through $\mathcal B_i(v)$ and exits $v$ through $\mathcal B_o(v)$, or
vice versa.

A \emph{weight system} on $\tau$ is a function $w:\mathcal B(\tau)\to \R_+$ satisfying the
switch equations: for any switch $v$ of $\tau$ we have
\[
\sum_{e\in \mathcal B_i(v)} w(e) = \sum_{f\in \mathcal B_o(v)} w(f).
\]
Associated to the pair $(\tau,w)$ we construct the following
\emph{union of foliated rectangles}. For each $b\in \mathcal B(\tau)$ we assign a
rectangle $R(b)=[0,w(b)]\times[0,1]$. 
We glue the rectangles at each switch $v$ as follows.
For any switch $v$ of $\tau$ we consider the
interval $I(v)=[0,\ell]$ where
\[
\ell = \sum_{b\in \mathcal B_i(v)} w(b) = \sum_{b \in \mathcal B_o(v)} w(b).
\]
Suppose that $b_1<\ldots <b_n$ are the
outgoing branches at $v$. Then $I(v)$ is divided into consecutive closed
intervals $I_1,\ldots,I_n$ of lengths $w(b_1),\ldots,w(b_n)$, respectively,
where $0\in I_1$ and the $I_i$'s overlap only on their boundaries.
Then we glue $[0,w(b_i)]\times\{0\}\subset R(b_i)$
via an orientation-preserving isometry to the interval $I_i$.
Similarly, $I(v)$ is also divided into intervals
$J_1,\ldots,J_m$ of lengths $w(c_1),\ldots,w(c_m)$ where $c_1,\ldots,c_m$ are the
incoming branches at $v$. Then we glue $[0,w(c_j)]\times \{1\}\subset R(c_j)$
to $J_j$ via an orientation-preserving isometry.
The union of foliated rectangles $\mathcal F$
for $(\tau,w)$ is the quotient
of the disjoint union of the rectangles $R(b)$ and intervals $I(v)$ by these gluing relations.

Each rectangle $R(b)$ of $\mathcal F$ is foliated by the vertical segments
$\{v\}\times [0,1]$ for $v\in [0,w(b)]$. 
This endows $\mathcal F$ with the structure
of a singular foliation. The singularities are the points where at least three
rectangles of $\mathcal F$ meet. 
In fact, at most four rectangles can meet, in which case we have two rectangles on both sides of an interval $I(v)$. By thickening $I(v)$ to a rectangle, we assume exactly three rectangles meet at each singularity.
A \emph{leaf} $\ell$ of $\mathcal F$ is an embedding of $\R$ into $\mathcal F$ that is the union of a sequence
\[
\ldots \sigma_{-1} \sigma_0 \sigma_1\ldots
\]
of vertical line segments $\sigma_i = \{v_i\}\times [0,1]\subset R(b_i)$ with consecutive segments meeting at endpoints, 
and which satisfies the following conditions. 
First, we require
that $\ldots b_{-1}b_0b_1\ldots$ is a train path of $\tau$.
Second, we have an additional requirement when the leaf contains at least two singularities, which we now describe. Given an orientation of a leaf $\ell$, singularities on $\ell$ fall into two types, merging or splitting; see Figure \ref{fig: train}. Moreover, singularities along $\ell$ must alternate between the two types as they arise from gluing of rectangles. 
At a merging singularity, there are two possible local pictures of $\ell$, namely merging from the left or right branch. Similarly, at a splitting singularity, $\ell$ splits to the left or right branch.
For a leaf $\ell$ containing at least two singularities, we require a choice of \emph{left} or \emph{right}: either $\ell$ always merges from the left and splits to the left, or it always merges from the right and splits to the right; see Figure \ref{fig: train}. 
\begin{figure}
    \centering
    \includegraphics[scale=0.8]{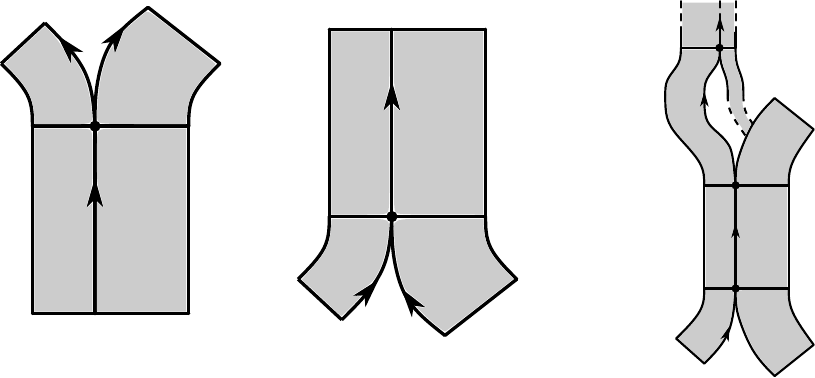}
    \caption{Left: two singular leaves that split after passing through a splitting singularity; middle: two singular leaves that merge after passing through a merging singularity; right: a singular leaf (indicated by the arrows) passes through three singularities and always splits to the left and merges from the left.}
    \label{fig: train}
\end{figure}
A leaf will be called \emph{singular} if it contains a singularity and \emph{non-singular} otherwise.
We will use unions of foliated rectangles in Section \ref{sec:laminations} to construct geodesic laminations on hyperbolic surfaces.

Finally, suppose that $X$ has an isolated puncture $p$. A \emph{ray} $\ell$ is any complete
simple geodesic asymptotic to $p$ on at least one end.
The ray $\ell$ is a \emph{loop} if it is
asymptotic to $p$ on both ends. A ray is \emph{filling} if it intersects every loop based
at $p$. Denote by $\mathcal R(X;p)$ the graph whose vertices are the rays based at $p$,
and whose edges join pairs of rays that are disjoint. By \cite{Simult}, the graph
$\mathcal R(X;p)$ consists of uncountably many connected components. Among these
components, exactly one is of infinite diameter and Gromov hyperbolic. The remaining components are
\emph{cliques} of rays. These cliques consist of pairwise disjoint rays,
each of which intersects every ray
not lying in the clique. A ray in such a clique connected component
is called \emph{high-filling}.
The set of cliques of high-filling rays is identified with the Gromov boundary
$\partial \mathcal R(X;p)$ \cite[Theorem 6.3.1]{Simult}.
If a ray is filling but not high-filling then it is called \emph{$2$-filling} \cite[Lemma 5.6.4]{Simult}.
Thus there is a trichotomy:
a ray is either not filling, $2$-filling, or high-filling. As such, $2$-filling rays can be
thought of as \emph{fake boundary points} for the graph $\mathcal R(X;p)$, with properties
mimicking those of high-filling rays.
Finally, note that a priori the graph $\mathcal R(X;p)$ depends on the particular hyperbolic
metric $X$. However, if $Y$ is a different first kind complete hyperbolic surface homeomorphic
to $X$, then there is a natural bijection between the rays on $X$ based at $p$ and the rays
on $Y$ based at $p$; see the end of Part 1 in \cite{Simult}.
This bijection preserves the property of a ray being a loop, $2$-filling,
high-filling, etc. Hence we may actually define
$\mathcal R(S;p)\vcentcolon = \mathcal R(X;p)$ where $S$ is the underlying
topological surface to $X$, and the graph $\mathcal R(S;p)$ is well-defined
independent of a particular first kind hyperbolic metric on $S$.
When $S$ is the plane minus a Cantor set then $2$-filling rays exist, and the construction can be applied to many other surfaces of infinite type \cite{2Fill}. Theorem \ref{theorem:infiniteclique} now confirms their existence on any infinite type surface with at least one isolated puncture.

\section{Laminations}
\label{sec:laminations}

We consider the train track $\tau$ illustrated in Figure \ref{figure:rotation}. The weights
of three branches are labeled for some $\alpha\in(0,1)$. The weights of the other branches are determined by these
via the switch equations.
Associated to these weights on $\tau$, we construct the
standard \emph{union of foliated rectangles} $F$. There is a single branch of weight $1$
in $\tau$ which gives rise to a rectangle $R$ of $F$. 

\begin{figure}[h]

	\centering

	\def\svgwidth{0.5\textwidth}
	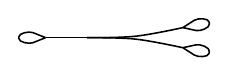

	\caption{A weighted train track $\tau$. The second return map to a horizontal interval in
		the rectangle $R$ of the union of foliated rectangles $F$ is a rotation by $\alpha$.}
	\label{figure:rotation}
\end{figure}

The train track $\tau$ has an infinite cyclic cover $\widetilde \tau$ which is pictured in
Figure \ref{figure:cover}. The weights on $\tau$ pull back to weights on
$\widetilde \tau$, some of which are labeled in Figure \ref{figure:cover}.

\begin{figure}[h]
	\labellist
	\small \hair 2pt
	\pinlabel $\frac{1-\alpha}{2}$ at 25 140
	\pinlabel $\frac{\alpha}{2}$ at 70 140
	\pinlabel $\frac{1-\alpha}{2}$ at 115 140
	\pinlabel $\frac{\alpha}{2}$ at 160 140
	\pinlabel $\frac{1-\alpha}{2}$ at 205 140
	\pinlabel $\frac{\alpha}{2}$ at 250 140
	\pinlabel $\frac{1-\alpha}{2}$ at 295 140
	\pinlabel $\frac{\alpha}{2}$ at 340 140
	\endlabellist
	\centering
	\includegraphics[scale=0.9]{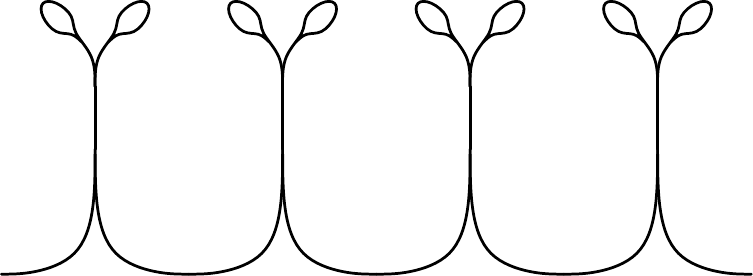}
	\caption{The infinite cyclic cover $\widetilde \tau$ of $\tau$.}
	\label{figure:cover}
\end{figure}
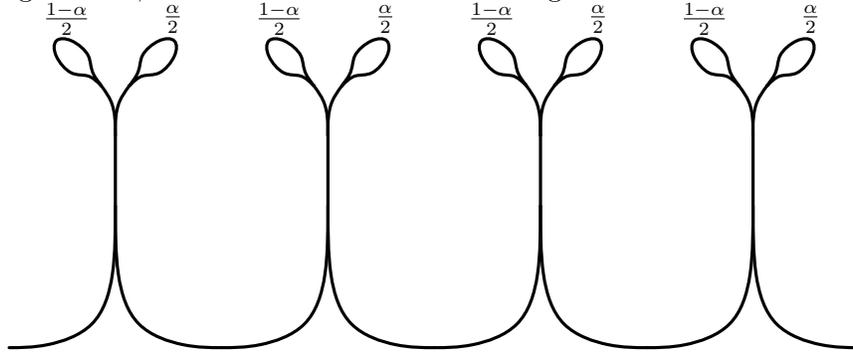

We consider the union of foliated rectangles $\widetilde F$ for $\widetilde \tau$ with
the described weights; see Figure \ref{figure: foliated}.
Then $\widetilde F$ is an infinite cyclic cover of $F$. The branches
of $\widetilde \tau$ of weight $1$ give rise to a sequence of rectangles
$\ldots, R_{-1}, R_0, R_1,\ldots$ in $\widetilde F$ which are indexed by $\Z$ and cover
the rectangle $R$ of $F$ corresponding to the branch of $\tau$ with weight $1$.
We choose the numbering so that there is a rectangle joining $R_i$ to $R_{i+1}$
and a rectangle joining $R_i$ to $R_{i-1}$ for each $i$, both with width $1/2$.
For the rectangle of width $1/2$ joining $R_i$ and $R_{i+1}$, one of its boundary leaves (the lower boundary in Figure \ref{figure: foliated}) extends to a singular leaf $\ell_i$ in $\tilde F$ passing through the singularities $s_i$ and $s_{i+1}$.
The leaves $\ell_{i-1}$ and $\ell_i$ share a ray $r_i$ starting at $s_i$.
The following consequence of Corollary \ref{cor: main} is crucial to our construction of an infinite clique of $2$-filling rays.
\begin{lemma}
	\label{lemma:denseleaves}
	Suppose that $\alpha\in (0,1)$ satisfies the conditions of Theorem
	\ref{theorem:skewprod}. Then for each $i$, the ray $r_i$ of $\widetilde F$	is dense in $\widetilde F$, hence so is the singular leaf $\ell_i$.
\end{lemma}
\begin{figure}
	\labellist
	\small \hair 2pt
	\pinlabel $s_{-1}$ at 105 10
	\pinlabel $s_0$ at 210 10
	\pinlabel $s_1$ at 315 10
	\pinlabel $s_2$ at 430 45
	\pinlabel $R_{-1}$ at 55 85
	\pinlabel $R_0$ at 165 85
	\pinlabel $R_1$ at 270 85
	\pinlabel $R_2$ at 375 85
	\endlabellist
	\centering
	\includegraphics[scale=0.8]{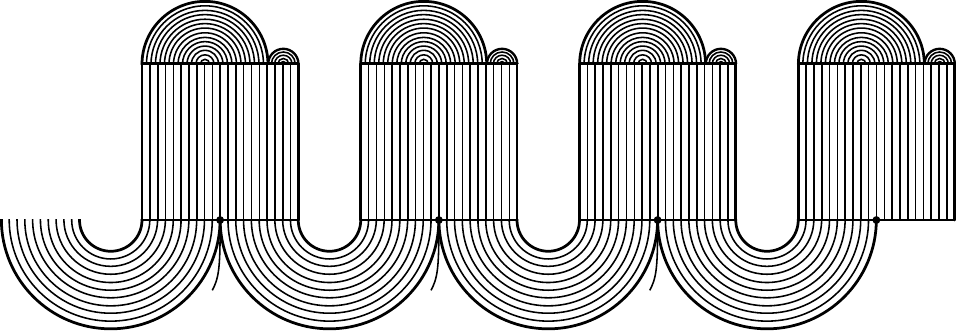}
	\caption{Part of the union of foliated rectangles $\widetilde{F}$}\label{figure: foliated}
\end{figure}
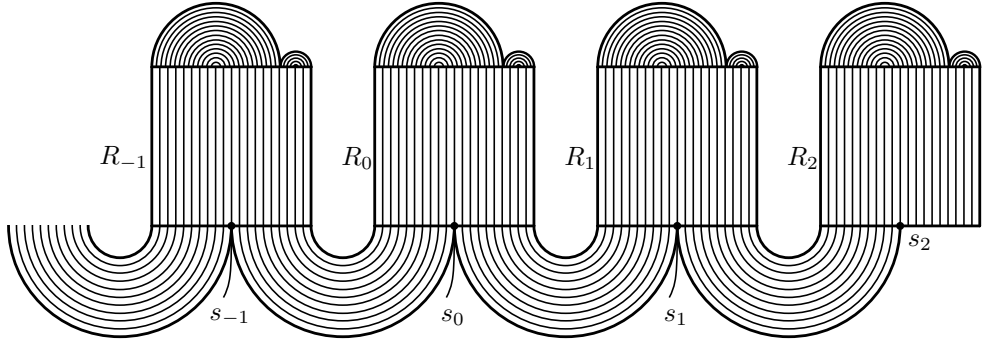
\begin{proof}
	To prove this, we parameterize the disjoint union $\bigcup_{i\in \Z} R_i$ by
	$[0,1]^2 \times \Z$, where $R_i$ is isometrically
	identified with the unit square $[0,1]^2$ and the leaves
	of $\widetilde F$ intersect the rectangles $[0,1]^2$ in vertical segments
	$\{x\}\times[0,1]$. We further choose the orientation on the vertical segments such that the singularity
	$s_i$ is given by $(1/2, 0, i)$ in these coordinates and each $r_i$ starts from $s_i$ by going upwards; see Figure \ref{figure: foliated}.
	
	After $r_i$ passes through $R_i$ for the first time, 
	it travels downwards in $R_i$ along the vertical segment with coordinate $(1-\alpha)-1/2 = 1/2-\alpha$ since $\alpha<1/2$ as in Theorem \ref{theorem:skewprod}.
	Thus $r_i$ passes through $(1/2-\alpha,0,i)$ to exit $R_i$. 
	At this point, $r_i$ will enter $R_{i-1}$ since $1/2-\alpha<1/2$, and it travels upwards starting at
	$(1-1/2+\alpha,0,i-1)=(1/2+\alpha,0,i-1)$.
	
	In general, if at some point the ray $r_i$ is traveling upwards in some $R_j$ along the vertical segment with coordinate $x\in(0,1)$, it hits the top of $R_j$ and then starts to travel downwards in $R_j$ along $\{b(x)\}\times[0,1]$, where 
	$$b(x)=\left\{ \begin{array}{ll}
		1-\alpha-x	& \text{if } x<1-\alpha,\\
		2-\alpha-x	& \text{if } x>1-\alpha.
	\end{array}\right.$$
	Note that $b(x)=1-t(x)$, where $t=t_\alpha$ is the rotation by $\alpha$ as in the introduction, that is, $t(x)\in (0,1)$ is the fractional part of $x+\alpha$. 
	At this point, $r_i$ exits $R_j$ at $(b(x),0,j)$ and enters $R_{j'}$ with $j'=j+1$ if $b(x)>1/2$ and $j'=j-1$ if $b(x)<1/2$. That is, $j'=j+f(t(x))$ for the function $f=\chi_{[0,1/2)}-\chi_{[1/2,1)}$ as in the definition of the transformation $T$ in Theorem \ref{theorem:skewprod}. Moreover, $r_i$ enters $R_{j'}$ by traveling upwards along the vertical segment with coordinate $1-b(x)=t(x)$.
	
	In the calculations above, we ignored all boundary cases since we only care about $x=t^n(1/2)$ for some $n\ge0$ and $\alpha$ is irrational.
	
	It follows from the analysis above that $r_i$ visits the $n$-th rectangle
 with entry point
	\[
	\big(t^{n-1}(1/2), 0, i + \sum_{j=1}^{n-1}f(t^j(1/2))\big).
	\]
	The sum $\sum_{j=1}^{n-1} f(t^j(1/2))$ is equal to $S_{n}(1/2) - f(1/2)=S_{n}(1/2)+1$ where
	$S_n$ is the $n$-th Birkhoff sum defined in the introduction. Thus, $r_i$ contains the points $\big(t^{n-1}(1/2),0, i+1 + S_{n}(1/2)\big)$ for $n\geq 1$, and the pairs
	$(t^{n-1}(1/2), i+1 +S_{n}(1/2))$ are dense in $[0,1]\times \Z$ by Corollary \ref{cor: main}.
	To see the last claim, note that for each $m\in\Z$, the set of $n\ge 0$ with $i+1+S_{n}(1/2)=m$ is $\Sigma(1/2,m-i-1)$, so the pairs above contain $(t^{n}(1/2), m)$ for all $n\in \Sigma(1/2,m-i-1)-1$ except possibly $n=0$, and the first coordinates of such pairs are dense by Corollary \ref{cor: main}.
	This proves that $r_i$ is dense in $R_m$ for any $m\in\Z$ and hence dense
	in the whole foliation $\widetilde F$. This completes the proof of Lemma \ref{lemma:denseleaves}.
\end{proof}

We now define a geodesic lamination on an infinite type hyperbolic surface. The track
$\widetilde \tau$ may be folded to yield the train track $\hat \tau$ pictured on the left of Figure
\ref{figure:folded}. The track $\hat \tau$ is carried by another track $\sigma$ shown on the right of Figure \ref{figure:folded}, which can in turn be embedded in any infinite type
surface $\Sigma$ with at least one isolated puncture $p$ as we explain below; 
see the left of Figure \ref{figure:proper}.

\begin{figure}[h]

	\centering

	\includegraphics[scale=0.8]{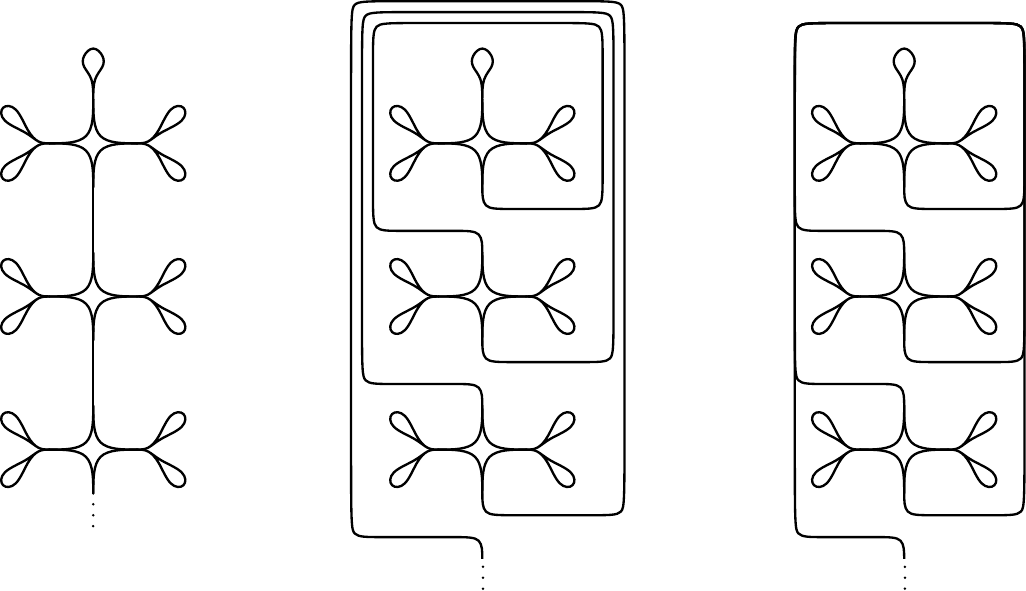}

	\caption{Left: The track $\hat \tau$ obtained by folding $\widetilde \tau$; Middle: Another embedding of $\widetilde \tau$ that spirals; Right: The track $\sigma$ obtained by collapsing parallel branches.}
	\label{figure:folded}
\end{figure}

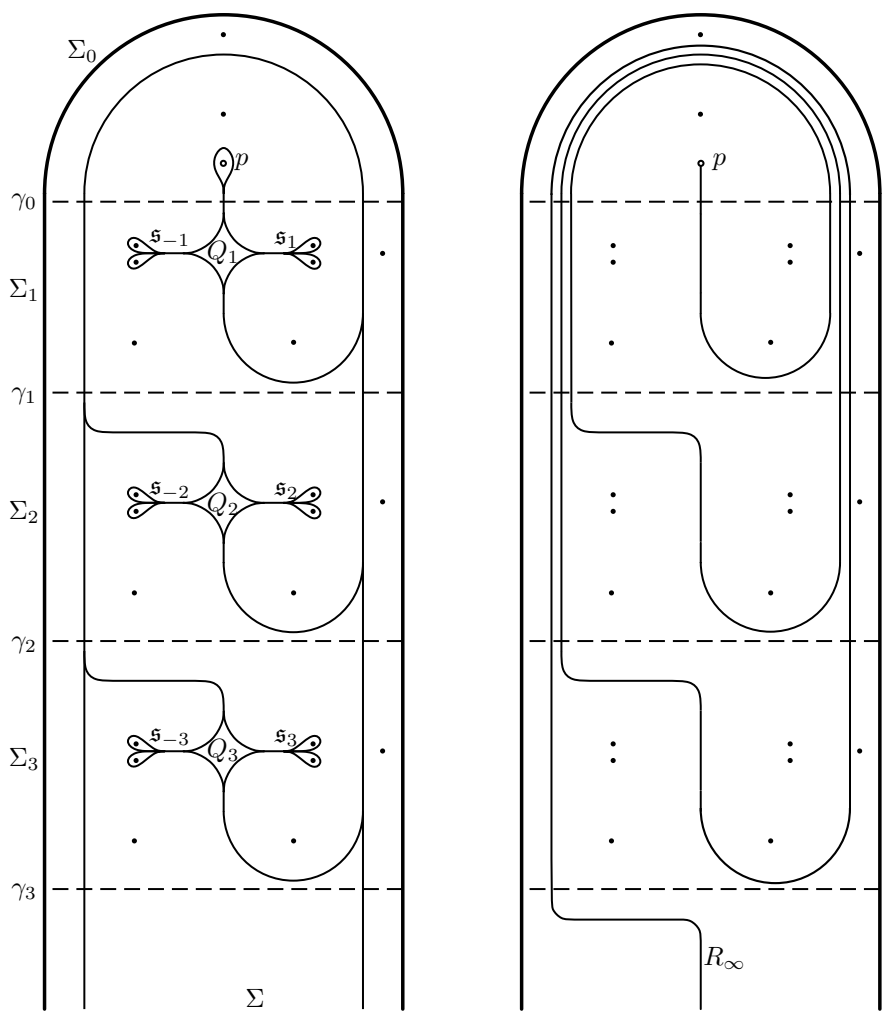
\begin{figure}[h]
	\labellist
	\small \hair 2pt
	
	\pinlabel $p$ at 75 320
	\pinlabel $\gamma_0$ at -7 305
	\pinlabel $\Sigma_0$ at 15 362
	
	\pinlabel $\mathfrak{s}_1$ at 92 291
	\pinlabel $\mathfrak{s}_{-1}$ at 48 291
	\pinlabel $Q_1$ at 68 286
	\pinlabel $\gamma_1$ at -7 232
	\pinlabel $\Sigma_1$ at -7 272
	
	\pinlabel $\mathfrak{s}_2$ at 92 196
	\pinlabel $\mathfrak{s}_{-2}$ at 48 196
	\pinlabel $Q_2$ at 68 191
	\pinlabel $\gamma_2$ at -7 138
	\pinlabel $\Sigma_2$ at -7 188
	
	\pinlabel $\mathfrak{s}_3$ at 92 103
	\pinlabel $\mathfrak{s}_{-3}$ at 48 103
	\pinlabel $Q_3$ at 68 98
	\pinlabel $\gamma_3$ at -7 45
	\pinlabel $\Sigma_3$ at -7 95
	\pinlabel $\Sigma$ at 80 5
	
	\pinlabel $p$ at 255 320
	\pinlabel $R_\infty$ at 257 20
	\endlabellist
	\centering
	\includegraphics{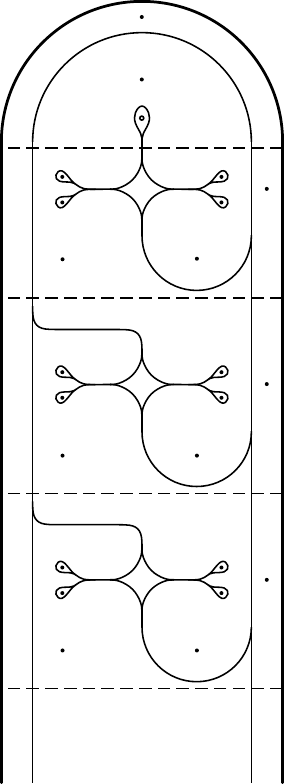}\hspace{40pt}
	\includegraphics{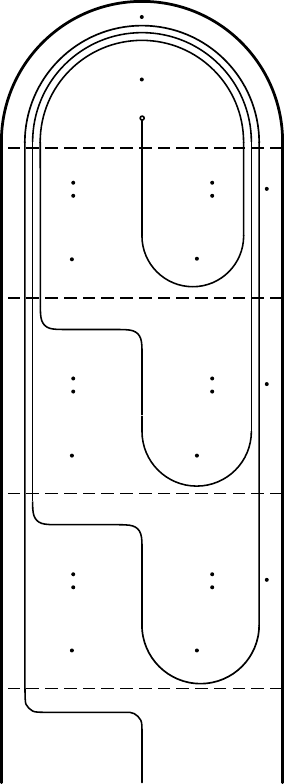}
	\caption{Left: The track $\sigma$ obtained by embedding $\hat \tau$ on an infinite
		type surface $\Sigma$ and collapsing parallel branches. Right: The non-filling
		ray $R_\infty$ on $\Sigma$.}
	\label{figure:proper}
\end{figure}

On the left of Figure \ref{figure:proper}, every
tiny black disk represents a subsurface of $\Sigma$ with a single boundary component, corresponding to the boundary of the black disk. We require each such subsurface to have either positive genus or at least two punctures. 
Furthermore, in $\Sigma$ the border line pictured on the left of Figure
\ref{figure:proper} is glued to itself by a reflection across the central vertical line, so that $\Sigma$ has no boundary.
With this identification, each dotted horizontal line segment represents an essential simple closed curve $\gamma_i$ on $\Sigma$.
Thus, the surface $\Sigma$ with the black disks removed is a flute surface $\Sigma'$.
Any infinite-type surface  without boundary and with at least one isolated puncture can be realized this way 
by appropriately choosing the topological type of the subsurfaces represented by the black disks.

\begin{lemma}\label{lemma: realize S}
    For any orientable surface $S$ of infinite type with at least one isolated puncture $p$, there is a sequence of surfaces $\{D_i\}_{i\ge1}$ each with one boundary component and either positive genus or at least two punctures, so that the surface $(\Sigma,p)$ above with the black disks homeomorphic to the $D_i$'s is homeomorphic to $(S,p)$.
\end{lemma}
\begin{proof}
    Recall the classification of (possibly non-compact) orientable surfaces without boundary \cite{Richards}. Each surface $S$ has a space of ends $E$, which is totally disconnected, compact, and metrizable.
    % what is a point in the end
    The non-planar ends form a closed subset $E_g\subset E$, which is nonempty if and only if $S$ has infinite genus. 
    Then the classification states that two surfaces are homeomorphic if and only if they have the same genus (possibly infinite) and the pairs of spaces of ends $(E,E_g)$ are homeomorphic. Moreover, given any pair $(E,E_g)$ with $E$ totally disconnected, compact, and metrizable and $E_g\subset E$ closed, and given $n\in\Z_{\ge0}\cup\{\infty\}$ so that $n<\infty$ iff $E_g=\emptyset$, there is an orientable surface $S$ with genus $n$ and spaces of ends homeomorphic to $(E,E_g)$.
    We say $S$ is of infinite type if either $E$ is an infinite set or $E_g$ is nonempty.

    Consider any $S$ of infinite type with an isolated puncture $p$, and denote its space of ends by $E$. Since $p$ is isolated, $E\setminus \{p\}$ is clopen. There are two cases:
    \begin{enumerate}
        \item Suppose $E$ is infinite. Then there is an accumulation point $x\in E$ and a sequence of nested clopen neighborhoods $E\setminus \{p\}=V_1\supset V_2\supset \cdots$ of $x$ with $\cap_i V_i=\{x\}$. Up to relabelling, we may assume that $U_i\defeq V_i\setminus V_{i+1}$ contains at least two points for all $i\ge 1$. For each $i\ge1$, there is a surface $S_i$ with space of ends homeomorphic to $(U_i, U_i\cap E_g)$. 
        Moreover, if $U_i\cap E_g=\emptyset$, then we can choose $S_i$ to have any genus $n_i\in\Z_{\ge0}$, which we now specify. If $S$ has finite genus, we may choose the $n_i$'s so that $\sum n_i$ is equal to the genus of $S$. 
        If $S$ has infinite genus, for each $i$ with $U_i\cap E_g= \emptyset$, we choose $n_i=0$ if $x$ is planar and $n_i>0$ if $x$ is non-planar. Let $D_i$ be $S_i$ with an open disk removed. 
        Then each $D_i$ either has positive genus or has at least two punctures.
        Choose the black disks in the construction of our surface $\Sigma$ above to be homeomorphic to the surfaces $D_i$. Then $(\Sigma,p)$ is homeomorphic to $(S,p)$ by the classification of surfaces.

        \item Suppose $E$ is finite. Then $E_g$ must be nonempty for $S$ to be of infinite type. Then $E=E_g\sqcup E'\sqcup\{p\}$ where $E'$ consists of planar ends other than $p$. For $1\le i\le |E_g|-1$, let $S_i$ be the surface of infinite genus and exactly one (non-planar) end, i.e. the Loch Ness monster. For $i=|E_g|$, let $S_i$ be the surface of genus one with $|E'|$ punctures. For $i>|E_g|$, let $S_i$ be the torus. Now for each $i\in\Z_+$, let $D_i$ be $S_i$ with an open disk removed. Choose the black disks in the construction of our surface $\Sigma$ above to be homeomorphic to the surfaces $D_i$. Then $\Sigma$ has infinite genus and has the same pair of spaces of ends as $S$, so again $(\Sigma,p)$ is homeomorphic to $(S,p)$ by the classification of surfaces.
    \end{enumerate}
\end{proof}

The simple closed curves $\{\gamma_i\}_{i\ge0}$ cut $\Sigma'$ into an infinite sequence of finite type subsurfaces $\{\Sigma_i\}_{i\ge0}$.
For each $i\ge 1$ (resp. $i=0$), the surface $\Sigma_i$ is bounded by two (resp. one) $\gamma_i$'s together with $7$ (resp. $2$) boundary components of black disks (resp. and a puncture $p$). Thus each $\Sigma_i$ admits a complete hyperbolic structure with geodesic boundary components all of length $1$.
In the sequel, we choose the metric on $\Sigma_i$ with the additional property that the (finitely many) train paths in $\sigma\cap \Sigma_i$ have lengths bounded above independent of $i$, 
which can be done for instance by making $\Sigma_i$ ($i\ge1$) all isometric under the obvious translation in Figure \ref{figure:proper}.
Since each black disk represents a surface with positive genus or at least two punctures, it admits a hyperbolic structure of the first kind so that the boundary is a geodesic of length $1$. 
Hence by gluing, we can endow $\Sigma$ with a complete hyperbolic metric of the first kind so that 
\begin{enumerate}
    \item each $\gamma_i$ is a closed geodesic of length $1$,
    \item the train paths in $\sigma\cap \Sigma_i$ have lengths bounded above independent of $i$.
\end{enumerate}

Let $\widetilde \Sigma\cong\Hbb^2$ be the universal cover of $\Sigma$. Consider the preimage $\widetilde \sigma$ of $\sigma$ and the collection $\mathcal L$ of lifts of all $\gamma_i$'s to $\widetilde \Sigma$.
We notice the following fact:
\begin{lemma}\label{lemma:straighten}
	For the choice of hyperbolic metric on $\Sigma$ above, there are uniform constants $K,C>0$ such that any bi-infinite train path of $\widetilde \sigma$ is a $(K,C)$-quasi-geodesic. In particular, it limits to two distinct points on the Gromov boundary $\partial \widetilde \Sigma$.
\end{lemma}
\begin{proof}
    We show this by looking at the intersections with lines in $\mathcal L$.
    Each lift of $\gamma_i$ is a bi-infinite geodesic in $\mathcal L$. 
    Note that there is a lower bound on the distance between any two lines of $\mathcal L$ by the collar lemma since the $\gamma_i$'s have bounded length.
    Moreover, the segment between any two consecutive intersections of the train path with $\mathcal L$ is a lift of a train path in $\sigma\cap \Sigma_i$ for some $i$, and thus its length is bounded above by a uniform constant due to our choice of metric.
    
    We claim that any train path of $\widetilde \sigma$ with endpoints on two lines of $\mathcal L$ is not homotopic, relative to endpoints, into a line in $\mathcal L$. Given the claim, any bi-infinite train path of $\widetilde \sigma$ intersects a bi-infinite non-backtracking sequence of geodesics in $\mathcal L$ at a uniformly bounded linear rate, from which the lemma follows.
    
    The claim above follows from the observations below. There are only two homeomorphism classes of pairs $(\Sigma_i, \Sigma_i \cap \sigma)$. Moreover, in each track $\Sigma_i \cap \sigma$, there
	are finitely many train paths. Finally, by the choice of $\gamma_i$'s, no train path of $\Sigma_i \cap \sigma$ is homotopic into $\partial \Sigma_i$ via a homotopy keeping the endpoints of the train path on the boundary, and no two distinct train paths of $\Sigma_i \cap \sigma$ are homotopic via such a homotopy. 
\end{proof}

Consequently any bi-infinite train path $t$ of $\widetilde \sigma$ may be straightened to a geodesic $\alpha$ in $\partial \widetilde \Sigma$ with the same endpoints on the Gromov boundary.
Moreover, the proof above implies that the sequences of lines in $\mathcal L$ intersecting $\alpha$ and $t$, respectively, are identical.
In addition, the intersections with $\mathcal L$ cut both $\alpha$ and $t$ into segments of length bounded above and below by uniform constants. 
We also see from the last observation made in the proof above that if $t_1,t_2$ are train paths in $\widetilde \sigma$ such that $t_i$ joins a line $L_i \in \mathcal L$ to a line $L_i' \in \mathcal L$, then $t_1$ and $t_2$
are equal if and only if $L_1=L_2$ and $L_1'=L_2'$. 
The following lemmas can be deduced from these facts.

\begin{lemma}\label{lemma: infinite intersection}
	Let $\{t_i\}_{i=1}^\infty$ be bi-infinite train paths of $\widetilde \sigma$
	straightening to geodesics $\{\alpha_i\}_{i=1}^\infty$ of $\widetilde \Sigma$.
	If $\alpha$ is a geodesic of $\widetilde \Sigma$ such that $\alpha_i\to \alpha$, then $\alpha$ intersects infinitely many lines in $\mathcal L$ at each end.
\end{lemma}
\begin{proof}
	Suppose this is not the case. Then one end of $\alpha$ projects to a geodesic ray in $\Sigma$ disjoint from the curves $\gamma_j$.
	Note that each $\alpha_i$ projected to $\Sigma$ is disjoint from the boundary curve of each black disk in Figure \ref{figure:proper}.
	Hence the limiting ray above is also disjoint from such boundary components. So the ray must be trapped in some $\Sigma_j$.
	This implies that there is an arbitrarily long geodesic segment $\beta_i$ inside $\Sigma_j$ in the projection of $\alpha_i$ for $i$ sufficiently large. This contradicts the observation we made above, that $\alpha_i$ is divided by $\mathcal L$ into segments of uniformly bounded length.
\end{proof}

\begin{lemma}\label{lemma:ttclosed}
	Let $\{t_i\}_{i=1}^\infty$ be bi-infinite train paths of $\widetilde \sigma$
	straightening to geodesics $\{\alpha_i\}_{i=1}^\infty$ of $\widetilde \Sigma$.
	Suppose that $\alpha$ is a geodesic of $\widetilde \Sigma$. Then $\alpha_i\to \alpha$
	if and only if $\alpha$ is also carried by $\widetilde \sigma$ and
	for any finite sub-path $s$ of the train path $t$ defining $\alpha$, 
	$s$ is contained in $t_i$ for all large enough $i$. 
	In particular, the set of geodesics carried by $\widetilde \sigma$ is closed
	in the space of geodesics of $\widetilde \Sigma$.
\end{lemma}
\begin{proof}
	We only focus on the less obvious direction: If $\alpha_i$ converges to a geodesic $\alpha$, then $\alpha$ is carried by $\widetilde \sigma$ 
	and for any finite sub-path $s$ of the train path $t$ defining $\alpha$, 
	$s$ is contained in $t_i$ for all large enough $i$.
	
	Each $\alpha_i$ intersects a bi-infinite sequence of lines in $\mathcal L$. By Lemma \ref{lemma: infinite intersection} and the fact that intersection is an open condition, 
	these sequences (with an appropriate choice of the $0$-th term) are pointwise eventually constant, with the limiting sequence equal to the lines in $\mathcal L$ intersecting $\alpha$. 
	As each straightening $\alpha_i$ intersects the same sequence of lines in $\mathcal L$ as the corresponding train path $t_i$ does,
	the limiting sequence above determines a train path $t$ carrying $\alpha$ with the desired properties.
\end{proof}

\begin{lemma}\label{lemma:commonendpt}
	Let $s$ and $t$ be bi-infinite train paths of $\widetilde \sigma$ straightening to
	geodesics $\beta$ and $\alpha$. Then $\beta$ and $\alpha$ share an endpoint
	$q \in \partial \widetilde \Sigma$ if and only if $s$ and $t$ share an infinite train path
	limiting to $q$.
\end{lemma}
\begin{proof}
	We again focus on the less obvious direction: If $\beta$ and $\alpha$ share an endpoint
	$q \in \partial \widetilde \Sigma$, then $s$ and $t$ share an infinite train path
	limiting to $q$. As in the proof above, $\alpha$ intersects the same bi-infinite sequence of lines in $\mathcal L$ as $t$ does.
	Note that the endpoints of these lines must converge to $q$ since lines in $\mathcal L$ are at distances uniformly bounded away from zero. 
	This implies that this sequence of lines would eventually all intersect the ray in $\beta$ limiting to $q$, and vice versa.
	It easily follows that $s$ and $t$ intersect the same sequence of lines in $\mathcal L$ at one end, determining the desired infinite train path.
\end{proof}

Now we define a geodesic lamination on $\Sigma$ as follows. Recall that the weights on $\tau$ induce weights on $\widetilde{\tau}$. 
Via the union of foliated rectangles construction, the leaves of the rectangles glue to a set of train paths on $\widetilde{\tau}$ and they correspond to a set $\widetilde{\mathcal T}$
of train paths on $\sigma$ via the carrying map. 
Finally, we consider the train path
$t_*$ in $\sigma$ which passes through each surface $\Sigma_i$ ($i>0$) exactly twice and never
returns. 
This is the train path parallel to the border line on the left side of Figure \ref{figure:proper}. 
We define
$\mathcal S\vcentcolon = \widetilde{\mathcal T} \cup \{t_*\}$, a set of train paths
on $\sigma$. By Lemma \ref{lemma:straighten}, we may straighten the train paths in
$\mathcal S$ to geodesics. Denote the resulting set of geodesics by $\Lambda$. Since
the train paths in $\mathcal S$ do not cross, neither do the geodesics of $\Lambda$.

\begin{lemma}\label{lemma:lambdaclosed}
	The set of geodesics $\Lambda$ is closed as a subset of $\Sigma$. Therefore it is a
	geodesic lamination on $\Sigma$.
\end{lemma}

We postpone the proof of Lemma \ref{lemma:lambdaclosed} and first discuss the train paths in $\mathcal S$ in more detail. 
Any nonsingular leaf of $\widetilde F$ uniquely determines a train path in $\sigma$, and it is determined by any segment of the leaf contained in a foliated rectangle of $\widetilde F$.
Now we describe train paths of $\widetilde{\mathcal T}$ corresponding to singular leaves of $\widetilde{F}$.
Note that there is a sequence of quadrilaterals $Q_i$ in the complement of the train track $\sigma$; see the left of Figure \ref{figure:proper}.
In each $Q_i$, $i\ge1$, a pair of singularities $\mathfrak s_i$ and $\mathfrak s_{-i}$ sit at the two opposite horizontal corners, corresponding to the singularities $s_i$ and $s_{-i+1}$ in Figure \ref{figure: foliated}. 
For each $i \in \Z_+$, there is a unique singular leaf $\ell_i$ containing $\mathfrak s_i$ and $\mathfrak s_{i+1}$ and a unique singular leaf $\ell_{-i}$ containing $\mathfrak s_{-i}$ and $\mathfrak s_{-i-1}$. 
Moreover, there is a singular leaf $\ell_0$ containing $\mathfrak s_{-1}$ and $\mathfrak s_1$. 
Thus, we have a collection $\ldots, \ell_{-1}, \ell_0, \ell_1, \ldots $
of singular leaves of $\widetilde F$ indexed by the integers, corresponding to the ones investigated in Lemma \ref{lemma:denseleaves}. 
This gives rise to a
collection $\ldots, t_{-1}, t_0,t_1,\ldots$ of train paths in $\mathcal S$. 
Note that for each $i\in\Z$, $\ell_i$ shares a ray with $\ell_{i-1}$
and $\ell_{i+1}$, respectively, so that $t_i$ shares a half-infinite sub-train-path with $t_{i-1}$ and $t_{i+1}$, respectively.

For each $i$, the train path $t_i$ gives rise to a leaf $L_i$ of $\Lambda$, corresponding to the leaf $\ell_i$ in $\widetilde{F}$ studied in Lemma \ref{lemma:denseleaves}.

\begin{proof}[Proof of Lemma \ref{lemma:lambdaclosed}]
    Lift $\Lambda$ to a set of geodesics $\widetilde{\Lambda}$ on $\widetilde \Sigma$.
    Let $\lambda_1,\lambda_2,\ldots$ be geodesics in $\widetilde{\Lambda}$ converging to the geodesic $\lambda$. By Lemma \ref{lemma:ttclosed}, $\lambda$ is carried by $\widetilde \sigma$. Let $t$ be a train path of $\widetilde \sigma$ straightening to $\lambda$. If $t$ projects to $t_*$ in $\Sigma$ then $\lambda$ is in $\widetilde \Lambda$ by definition. Otherwise, $t$ passes through some branch $b$ on the boundary of a quadrilateral in the complement of $\widetilde \sigma$. Hence, $t_i$ passes through $b$ for $i$ sufficiently large. Without loss of generality we assume that $t_i$ passes through $b$ for every $i$.

    Lift $\widetilde F$ to a foliation of a subset of $\widetilde \Sigma$. Then for each $i$, $t_i$ is the train path defined by a (possibly singular) leaf of $\widetilde F$. That is, after collapsing the rectangles of $\widetilde F$ to branches, and composing with the carrying map to $\widetilde \sigma$, we obtain $t_i$. Thus, for each $t_i$ there is a unique vertical line segment $u_i$ in the rectangle $R(b)$ which the leaf corresponding to $t_i$ passes through. Up to passing to a subsequence, the segments $u_i$ converge to a vertical segment $u$ in $R(b)$, and we may further assume all segments $u_i$ lie on the same side of $u$, say the left side (for a chosen orientation of $u$). Let $r_i$ be the (possibly singular) leaf defining $t_i$ (which passes through $u_i$). Furthermore, let $r$ be the (possibly singular) leaf which passes through $u$, and, when given the orientation induced by $u$, merges from or splits into the left rectangle at every singularity that it passes through (if there are indeed any such singularities). Then $u_i$ converges to $u$, since for any finite sequence of rectangles that $u$ passes through, $u_i$ also passes through the same sequence for $i$ sufficiently large. Hence $t_i$ converges to the train path defined by $u$, which is therefore equal to $t$. In particular, $\lambda$ is a leaf of $\widetilde \Lambda$.
\end{proof}

The following lemma is the key to obtain our infinite clique of $2$-filling rays.
\begin{lemma}\label{lemma: shape of complementary component}
	The complementary component to $\Lambda$ containing $p$ is a once-punctured
	ideal polygon with countably infinitely many ends, exactly one of which is the limit of
	the others. The sides of the ideal polygon are the leaves $\{L_i\}_{i\in\Z}$.
\end{lemma}

In order to prove Lemma \ref{lemma: shape of complementary component}, we look at some particular rays, which we denote as $R_i$ with $i \in\Z\cup\{\infty\}$. We apologize for reusing the notation and warn the reader not to confuse them with the rectangles $R_i$ discussed earlier (which play no role in the rest of the paper).

First we define a ray $R_\infty$ with one end at the isolated puncture $p$. This is the geodesic ray pictured on the right of Figure \ref{figure:proper}. 
It is non-proper and not filling (as defined at the end of Section \ref{sec:background}).
Observe that it is disjoint from $\Lambda$.

We also define a sequence of geodesic rays $R_i$ for $i\in \Z \setminus \{0\}$ as follows.
The ray $R_i$ is obtained by following $R_\infty$ until it enters the quadrilateral $Q_i$ with
ends corresponding to the singularities $\mathfrak s_{\pm i}$. Thereafter, it passes through
$\mathfrak s_i$ and follows the common half-infinite sub-train-path of $t_{i-1}$ and
$t_i$ (if $i>0$) or of $t_i$ and $t_{i+1}$ (if $i<0$). 
The path just described may be homotoped to be simple
and disjoint from any given leaf of $\Lambda$, and furthermore, none of its arcs in a
component of $\Sigma\setminus \sqcup_i \gamma_i$ is homotopic into the boundary. 
Hence the path may be straightened to a geodesic ray $R_i$; 
see Figure \ref{figure:rays}.

\begin{figure}[h]

	\centering

	\def\svgwidth{0.7\textwidth}
	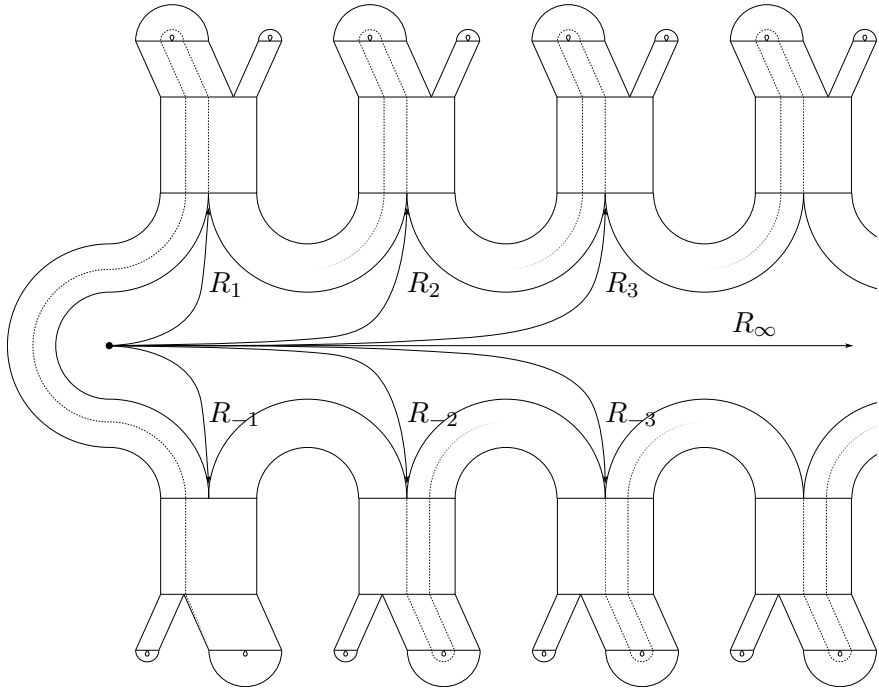

	\caption{The rays $R_i$ and $R_\infty$. They enter the union of foliated rectangles at
		a cusp and thereafter follow a leaf of the corresponding singular foliation (the
		dotted lines in the figure). For ease of presentation, the picture has been
		``unwrapped'' before being embedded into $\Sigma$.}
	\label{figure:rays}
\end{figure}

Now we prove Lemma \ref{lemma: shape of complementary component}.
\begin{proof}
	Recall that the surfaces $\Sigma_j$ are the complementary subsurfaces
	of the dotted curves $\gamma_j$ and black disks in Figure \ref{figure:proper}. 
	We see that for $i\gg 0$ or $i\ll 0$, $R_i$ and $R_\infty$ pass through
	many of the surfaces $\Sigma_j$ in the same order, and in each such $\Sigma_j$, the arcs of
	$R_i$ and $R_\infty$ are homotopic, keeping endpoints on the boundary. Hence we have
	$R_i\to R_\infty$ as $i\to \infty$ and as $i\to -\infty$. Furthermore, $R_i$ is asymptotic
	to $L_i$ and $L_{i-1}$ if $i>0$ and $R_i$ is asymptotic to $L_i$ and $L_{i+1}$ if $i<0$.
	Finally, the rays $R_i$ occur in the order
	\[
	 R_\infty < \ldots < R_2 < R_1 < R_{-1} < R_{-2} < \ldots < R_\infty
	\]
	in the circular order on geodesics asymptotic to $p$.
	
	Lift $R_\infty$ to a ray $\widetilde{R}_\infty$ in the universal cover $\widetilde \Sigma$.
	Then $\widetilde{R}_\infty$ has one endpoint at a lift $\widetilde p$ of $p$ on $\partial \widetilde\Sigma$ and
	the other endpoint at a point $z\in \partial \widetilde\Sigma$. Let $g$ be a generator of the cyclic
	subgroup of $\pi_1(\Sigma)$ fixing $\widetilde p$. There is a unique lift $\widetilde{R}_i$ of $R_i$ based at
	$\widetilde p$, between $\widetilde{R}_\infty$ and $g\cdot \widetilde{R}_\infty$. Up to
	replacing $g$ by $g^{-1}$, we see that $\widetilde{R}_i \to g\widetilde{R}_\infty$ as
	$i\to +\infty$ and $\widetilde{R}_i \to \widetilde{R}_\infty$ as $i\to -\infty$. Moreover,
	there is a lift $\widetilde{L}_i$, for $i\in \Z$, such that
	\begin{itemize}
		\item $\widetilde{L}_0, \widetilde{R}_{-1}, \widetilde{R}_1$ are the sides of an ideal
		triangle;
		\item $\widetilde{L}_i, \widetilde{R}_i, \widetilde{R}_{i+1}$ are the sides of an ideal
		triangle for $i>0$;
		\item $\widetilde{L}_i, \widetilde{R}_i, \widetilde{R}_{i-1}$ are the sides of an ideal
		triangle for $i<0$.
	\end{itemize}
	
	Consequently, $\widetilde{R}_\infty$, $g\widetilde{R}_\infty$, and the $\widetilde{L}_i$'s
	form the sides of a polygon with countably infinitely many ends. Each of these ends
	is isolated except $z$ and $gz$, which are limits of the others. After quotienting by $g$,
	we obtain a once-punctured ideal polygon with a countable set of ends, exactly one of which
	is the limit of the others, as claimed.
\end{proof}

Finally we prove Theorems \ref{theorem:lamination} and \ref{theorem:infiniteclique}.
\begin{proof}[Proof of Theorem \ref{theorem:infiniteclique}]
	By Lemma \ref{lemma: realize S}, any infinite type surface $S$ with at least one isolated puncture $p$ is homeomorphic to our surface $\Sigma$ by a suitable choice in the construction.
	Construct the lamination $\Lambda$ and use the notation as above.
	We see from Lemma \ref{lemma: shape of complementary component} that every simple ray based at $p$, except
	$R_i$ for $i\in \{\infty\}\cup\Z \setminus \{0\}$, intersects $L_i$ for some $i\in \Z$. 
	Moreover, $L_i$ has a half leaf asymptotic to $R_i$, and this half leaf is dense in $\Lambda$ by Lemma \ref{lemma:denseleaves}. Hence each $R_i$ accumulates onto $\Lambda$. 
	Thus, every simple ray based at $p$, not contained in
	$\{R_i\}_{i\in \{\infty\}\cup \Z \setminus \{0\}}$ intersects $R_i$ for every
	$i\in \Z\setminus \{0\}$. Thus each ray $R_i$ is $2$-filling, only disjoint from one non-filling long ray $R_\infty$,
	and $\{R_i\}_{i\in \Z \setminus \{0\}}$ is an infinite clique of $2$-filling rays.
\end{proof}

\begin{proof}[Proof of Theorem \ref{theorem:lamination}]
	Let $S=\Sigma$ with the hyperbolic structure and lamination $\Lambda$ as above. In the construction of $\Lambda$, choose $\alpha=[0;2m+1,2m+2,2m+2,\cdots]$ as in Example \ref{example: non dense full orbit} for some $m\ge2$, which satisfies the assumptions of Theorem \ref{theorem:skewprod}.
	Each leaf $L_i$ described above is dense in $\Lambda$ by Lemma \ref{lemma:denseleaves}, so $\Lambda$ is topologically transitive.
	On the other hand, the full orbit we examined in Example \ref{example: non dense full orbit} has (forward and backward) Birkhoff sum always non-positive, so the corresponding leaf misses infinitely many rectangles, and thus it is not dense. There is an obvious $\Z$ action on $\widetilde F$ in Figure \ref{figure: foliated}. It is straightforward to see that the $\Z$-orbit of this leaf yields infinitely many distinct non-dense leaves, which completes our proof.
\end{proof}

\bibliographystyle{plain}
\bibliography{rotations}

\end{document}

%% file: rotation_track.pdf_tex
%% Accompanies image file 'rotation_track.pdf' (pdf, eps, ps)
\begingroup%
  \makeatletter%
  \providecommand\color[2][]{%
    \errmessage{(Inkscape) Color is used for the text in Inkscape, but the package 'color.sty' is not loaded}%
    \renewcommand\color[2][]{}%
  }%
  \providecommand\transparent[1]{%
    \errmessage{(Inkscape) Transparency is used (non-zero) for the text in Inkscape, but the package 'transparent.sty' is not loaded}%
    \renewcommand\transparent[1]{}%
  }%
  \providecommand\rotatebox[2]{#2}%
  \newcommand*\fsize{\dimexpr\f@size pt\relax}%
  \newcommand*\lineheight[1]{\fontsize{\fsize}{#1\fsize}\selectfont}%
  \ifx\svgwidth\undefined%
    \setlength{\unitlength}{109.55356628bp}%
    \ifx\svgscale\undefined%
      \relax%
    \else%
      \setlength{\unitlength}{\unitlength * \real{\svgscale}}%
    \fi%
  \else%
    \setlength{\unitlength}{\svgwidth}%
  \fi%
  \global\let\svgwidth\undefined%
  \global\let\svgscale\undefined%
  \makeatother%
  \begin{picture}(1,0.33007337)%
    \lineheight{1}%
    \setlength\tabcolsep{0pt}%
    \put(0,0){\includegraphics[width=\unitlength,page=1]{rotation_track.pdf}}%
    \put(0.97,0.1){\color[rgb]{0,0,0}\makebox(0,0)[lt]{\lineheight{1.25}\smash{\begin{tabular}[t]{l}\textit{$\alpha/2$}\end{tabular}}}}%
    \put(0.95,0.22486978){\color[rgb]{0,0,0}\makebox(0,0)[lt]{\lineheight{1.25}\smash{\begin{tabular}[t]{l}\textit{$(1-\alpha)/2$}\end{tabular}}}}%
    \put(0,0){\includegraphics[width=\unitlength,page=2]{rotation_track.pdf}}%
    \put(0.38204715,0.2){\color[rgb]{0,0,0}\makebox(0,0)[lt]{\lineheight{1.25}\smash{\begin{tabular}[t]{l}\textit{$1$}\end{tabular}}}}%
  \end{picture}%
\endgroup%

%% file: rays.pdf_tex
%% Accompanies image file 'rays.pdf' (pdf, eps, ps)
\begingroup%
  \makeatletter%
  \providecommand\color[2][]{%
    \errmessage{(Inkscape) Color is used for the text in Inkscape, but the package 'color.sty' is not loaded}%
    \renewcommand\color[2][]{}%
  }%
  \providecommand\transparent[1]{%
    \errmessage{(Inkscape) Transparency is used (non-zero) for the text in Inkscape, but the package 'transparent.sty' is not loaded}%
    \renewcommand\transparent[1]{}%
  }%
  \providecommand\rotatebox[2]{#2}%
  \newcommand*\fsize{\dimexpr\f@size pt\relax}%
  \newcommand*\lineheight[1]{\fontsize{\fsize}{#1\fsize}\selectfont}%
  \ifx\svgwidth\undefined%
    \setlength{\unitlength}{3566.59048967bp}%
    \ifx\svgscale\undefined%
      \relax%
    \else%
      \setlength{\unitlength}{\unitlength * \real{\svgscale}}%
    \fi%
  \else%
    \setlength{\unitlength}{\svgwidth}%
  \fi%
  \global\let\svgwidth\undefined%
  \global\let\svgscale\undefined%
  \makeatother%
  \begin{picture}(1,0.78540563)%
    \lineheight{1}%
    \setlength\tabcolsep{0pt}%
    \put(0,0){\includegraphics[width=\unitlength,page=1]{rays.pdf}}%
    \put(0.23138241,0.45408651){\color[rgb]{0,0,0}\makebox(0,0)[lt]{\lineheight{1.25}\smash{\begin{tabular}[t]{l}\textit{$R_1$}\end{tabular}}}}%
    \put(0.45962112,0.45408651){\color[rgb]{0,0,0}\makebox(0,0)[lt]{\lineheight{1.25}\smash{\begin{tabular}[t]{l}\textit{$R_2$}\end{tabular}}}}%
    \put(0.68785983,0.45408651){\color[rgb]{0,0,0}\makebox(0,0)[lt]{\lineheight{1.25}\smash{\begin{tabular}[t]{l}\textit{$R_3$}\end{tabular}}}}%
    \put(0.2313824,0.3052835){\color[rgb]{0,0,0}\makebox(0,0)[lt]{\lineheight{1.25}\smash{\begin{tabular}[t]{l}\textit{$R_{-1}$}\end{tabular}}}}%
    \put(0.45962112,0.3052835){\color[rgb]{0,0,0}\makebox(0,0)[lt]{\lineheight{1.25}\smash{\begin{tabular}[t]{l}\textit{$R_{-2}$}\end{tabular}}}}%
    \put(0.68785983,0.3052835){\color[rgb]{0,0,0}\makebox(0,0)[lt]{\lineheight{1.25}\smash{\begin{tabular}[t]{l}\textit{$R_{-3}$}\end{tabular}}}}%
    \put(0.83325681,0.40832399){\color[rgb]{0,0,0}\makebox(0,0)[lt]{\lineheight{1.25}\smash{\begin{tabular}[t]{l}\textit{$R_\infty$}\end{tabular}}}}%
    \put(0,0){\includegraphics[width=\unitlength,page=2]{rays.pdf}}%
  \end{picture}%
\endgroup%